\documentclass[12pt]{amsart}

\usepackage[all]{xy}
\usepackage{fullpage}
\usepackage{latexsym}
\usepackage{amsmath}
\usepackage{amsfonts}
\usepackage{amssymb}
\usepackage{amsthm}
\usepackage{eucal}
\usepackage{enumerate,yfonts}
\usepackage{mathrsfs}
\usepackage{graphicx}
\usepackage{graphics}
\usepackage{epstopdf}
\usepackage{amscd}
\usepackage{bbm}
\usepackage{hyperref}
\usepackage{url}
\usepackage{color}
\usepackage{bbm}
\usepackage{cancel}
\usepackage{enumerate}
\usepackage{amsmath,amsthm}
\usepackage{amssymb}
\usepackage{epsfig}
\usepackage{pstricks}
\usepackage{xy}
\usepackage{xypic}
\usepackage{bbm}

\newtheorem{thm}{Theorem}[section]
\newtheorem{corollary}[thm]{Corollary}
\newtheorem{lemma}[thm]{Lemma}

\newtheorem{proposition}[thm]{Proposition}
\newtheorem{prop}[thm]{Proposition}
\newtheorem{conjecture}[thm]{Conjecture}
\newtheorem{thm-dfn}[thm]{Theorem-Definition}

\theoremstyle{definition}
\newtheorem{definition}[thm]{Definition}
\newtheorem{remark}[thm]{Remark}
\newtheorem{example}[thm]{Example}

\numberwithin{equation}{section}

\newcommand{\fg}{{\mathfrak g}}
\newcommand{\ft}{{\mathfrak t}}
\newcommand{\fl}{{\mathfrak l}}
\newcommand{\fb}{{\mathfrak b}}
\newcommand{\fu}{{\mathfrak u}}
\newcommand{\fp}{{\mathfrak p}}

\newcommand{\fa}{{\mathfrak a}}

\newcommand{\fc}{{\mathfrak c}}

\newcommand{\rW}{{\mathrm W}}

\newcommand{\bC}{{\mathbb C}}

\newcommand{\bZ}{{\mathbb Z}}

\newcommand{\mS}{\mathcal{S}}

\newcommand{\mF}{\mathcal{F}}

\newcommand{\mO}{\mathcal{O}}
\newcommand{\mL}{\mathcal{L}}

\newcommand{\on}{\operatorname}

\newcommand{\lra}{\longrightarrow}

\newcommand{\bs}{\backslash}
\newcommand{\is}{\simeq}

\newcommand{\Loc}{\on{LocSys}}

\newcommand{\quash}[1]{}  %%Anything in \quash is ignored
\newcommand{\nc}{\newcommand}

\newcommand{\frakg}{{\mathfrak g}}

\newcommand{\frakk}{{\mathfrak k}}
\newcommand{\frakl}{{\mathfrak l}}

\newcommand{\frakp}{{\mathfrak p}}

\newcommand{\frakr}{{\mathfrak r}}
\newcommand{\fraks}{{\mathfrak s}}

\newcommand{\fraku}{{\mathfrak u}}

\newcommand{\frakM}{{\mathfrak M}}

\newcommand{\bbC}{{\mathbb C}}

\newcommand{\bbR}{{\mathbb R}}

\newcommand{\bbZ}{{\mathbb Z}}

\newcommand{\calB}{{\mathcal B}}

\newcommand{\calF}{{\mathcal F}}

\newcommand{\calH}{{\mathcal H}}

\newcommand{\calN}{{\mathcal N}}
\newcommand{\calO}{{\mathcal O}}

\newcommand{\calS}{{\mathcal S}}

\newcommand{\calZ}{{\mathcal Z}}

\newcommand{\bfM}{{\textbf M}}

\nc{\al}{{\alpha}} \nc{\be}{{\beta}} \nc{\ga}{{\gamma}}
\nc{\ve}{{\varepsilon}} \nc{\Ga}{{\Gamma}} %\nc{\la}{{\lambda}}
\nc{\La}{{\Lambda}}

\nc{\ad }{{\on{ad }}}

\nc{\aff}{{\on{aff}}} \nc{\Aff}{{\mathbf{Aff}}}

\nc{\der}{{\on{der}}}

\nc{\diag}{{\on{diag}}}
\newcommand{\End}{{\on{End}}}
\nc{\Fl}{{\calF\ell}}

\nc{\Hg}{{\on{Higgs}}}
\newcommand{\Hom}{{\on{Hom}}}
\newcommand{\id}{{\on{id}}}
\nc{\Id}{{\on{Id}}}

\nc{\Ind}{{\on{Ind}}}

\nc{\Op}{{\on{Op}}}

\nc{\res}{{\on{res}}}

\newcommand{\Spec}{{\on{Spec}}}

\nc{\tr}{{\on{tr}}}

\newcommand{\GL}{{\on{GL}}}
\nc{\GSp}{{\on{GSp}}} \nc{\GU}{{\on{GU}}} \nc{\SL}{{\on{SL}}}
\nc{\SU}{{\on{SU}}} \nc{\SO}{{\on{SO}}}

\nc{\nh}{{\Loc_{J^p}(\tau')}}
\nc{\bnh}{{\Loc_{\breve J^p}(\tau')}}

\nc{\bU}{{\overline{U}}} \nc{\IC}{{\on{IC}}}

\newcommand{\dota}{\ddot{\text a}}

\newcommand{\beqn}{\begin{equation*}}
\newcommand{\eeqn}{\end{equation*}}

\newcommand{\beq}{\begin{equation}}
\newcommand{\eeq}{\end{equation}}

\newcommand{\frakgl}{{\frakg\frakl}}

\nc{\QM}{QM}
\nc{\eval}{\textup{ev}}

\quash{
\topmargin-0.5cm \textheight22cm \oddsidemargin1.2cm \textwidth14cm}
\begin{document}
\title{Real and symmetric matrices}

       \author{Tsao-Hsien Chen} 
        
       \address{School of Mathematics, University of Minnesota, Minneapolis, Vincent Hall, MN, 55455}
       \email{chenth@umn.edu}
       \author{David Nadler} 
        
        \address{Department of Mathematics, UC Berkeley, Evans Hall,
Berkeley, CA 94720}
        \email{nadler@math.berkeley.edu}
       
\maketitle
\begin{abstract}
We construct a family of involutions on the space 
$\frakgl_n'(\bC)$ of $n\times n$ matrices with real eigenvalues 
interpolating the complex conjugation 
and the transpose.  
We deduce from it a
stratified homeomorphism
between 
the space $\frakgl_n'(\bbR)$ of $n\times n$ real matrices
with real eigenvalues and the space $\fp_n'(\bC)$
of 
$n\times n$ symmetric matrices
with real eigenvalues, which restricts to a real analytic isomorphism between 
individual $\on{GL}_n(\bbR)$-adjoint orbits and 
$\on{O}_n(\bC)$-adjoint orbits.
We also establish similar results in more general settings of 
Lie algebras of classical types and quiver varieties.
To this end, we prove a general result about
involutions on hyper-K$\dota$hler quotients of linear spaces. 
We provide applications to the (generalized) Kostant-Sekiguchi correspondence,
singularities of real and symmetric adjoint orbit closures,
and Springer theory for real groups and symmetric spaces.

\end{abstract}
\tableofcontents

\newcommand{\bA}{{\mathbb A}}

\nc{\qQM}{\mathscr{QM}}

\nc{\sw}{\textup{sw}}
\nc{\inverse}{\textup{inv}}
\nc{\negative}{\textup{neg}}

\nc{\Sym}{\textup{Sym}}
\nc{\sym}{\mathit{sym}}
\nc{\Orth}{\textup{O}}
\nc{\Maps}{\textup{Maps}}
\nc{\sm}{\textup{sm}}

\nc{\image}{\textup{image}}

\nc{\risom}{\stackrel{\sim}{\to}}

\section{Introduction}

\subsection{Main results}
A key structural result in Lie theory is Cartan's classification of real forms
of a complex reductive Lie algebra $\fg$ in terms of 
holomorphic involutions.
It amounts to a bijection 
\beq\label{bijection}
\{\text{complex conjugations $\eta$ of}\ \fg\}/\text{isom}\longleftrightarrow\{\text{holomorphic involutions $\theta$ of }\fg\}/\text{isom}
\eeq
between isomorphism classes of complex conjugations and holomorphic involutions of 
$\fg$.
For example, in the case $\fg=\frak{gl}_n(\bbC)$, 
the complex conjugation $\eta(M)=\overline M$ 
with real form consisting of real matrices $\fg_\bbR=\frak{gl}_n(\bbR)$ 
corresponds to
the involution $\theta(M)=-M^t$ with 
$(-\theta)$-fixed points consisting of symmetric matrices $\fp=\frak p_n(\bC)$.
The interplay between the real $\fg_\bbR$ and symmetric $\fp$ pictures plays a fundamental role in the
structure and representation theory of real groups,
going back at least to Harish-Chandra's formulation of 
the representation theory of real groups in terms of 
$(\fg,K)$-modules.

One of the goals of the paper is to get a better understanding 
of Cartan's bijection and also the real and symmetric pictures for real groups
from the geometric point of view. 
To this end, let 
$\eta$ be a conjugation on $\fg$ and let $\theta$ be the corresponding involution under~\eqref{bijection}.
For simplicity, we assume $\eta$ is the split conjugation.
Then the subspace $\fg'$ of $\fg$ consisting of 
elements with real eigenvalues\footnote{Elements $x\in\fg$ such that the the adjoint action $\on{ad}_x:\fg\to\fg$ has only real eigenvalues.} is preserved by 
both $\eta$ and $-\theta$ and the first main result in the paper, Theorem \ref{intro: main 2}, is a construction 
of a real analytic family of involutions on $\fg'$
\beq\label{main:involutions}
\alpha_a:\fg'\lra\fg'\ \ \ \ a\in\on[0,1]
\eeq
interpolating the conjugation $\eta$ and the holomorphic involution $-\theta$,
that is, we have $\alpha_0=\eta$ and $\alpha_1=-\theta$, in the case 
when $\fg$ is of classical type.  
Using the family of involutions above
we prove the second main result of the paper, Theorem \ref{intro: main}, which says that 
there exists a stratified homeomorphism 
\beq\label{main:homeo}
\xymatrix{
\fg_\bbR' \ar[r]^-\sim& \fp'
}
\eeq
between the $\eta$ and ($-\theta$)-fixed points on
$\fg'$ compatible with various structures.\footnote{It is necessary to consider the subspace 
$\fg'\subset\fg$ but not the whole Lie algebra $\fg$ in the main results because in general the fixed points $\fg_\bbR=\fg^{\eta}$ and $\fp=\fg^{-\theta}$  have 
different dimensions and hence can't be homeomorphic to each other.} 
The family of involutions in~\eqref{main:involutions}
and the homeomorphism \eqref{main:homeo} can be thought as geometric refinements  
of 
Cartan's bijection \eqref{bijection}.

We deduce  several applications from the main results.
Assume $\fg$ is of classical types.
In Corollary \ref{intro: main 3}, We show that there exists a 
stratified homeomorphism 
\[
\xymatrix{
\calN_\bbR \ar[r]^-\sim& \calN_\fp
}
\]
between the real nilpotent cone $\calN_\bbR\subset\fg_\bbR$
and the symmetric nilpotent cone $\calN_\fp\subset\fp$
providing a lift of the celebrated Kostant-Sekiguchi correspondence between
real and symmetric nilpotent orbits.
In particular, it implies that $\calN_\bbR$ and $\calN_\fp$
have the same singularities, answering an open question (see, e.g., \cite[p354]{He}).
In Corollary \ref{formula for nearby cycles}, we show that Grinberg's nearby cycles sheaf on $\calN_\fp$
is isomorphic to the real Springer sheaf
given by the push-forward of the constant sheaf along the real Springer map, establishing a conjecture of 
Vilonen-Xue and the first author. 

The key ingredients in the proof are the
hyper-K\"ahler $\on{SU}(2)$-action on 
the space matrices arising from the quiver variety description in \cite{KP,KS,M,MV,Nak1}, and 
a general result about involutions on hyper-K\"ahler quotients of linear 
spaces (see Theorem \ref{intro: main 4}). The techniques used in the proof 
 are not specific to matrices and 
are applicable to a more general setting. 
For example, we also establish 
a quiver variety version of the main results.

We now describe the paper in more details.

%%%%%%%%%%%
\quash{
The classification goes as follows.
Let $\fg_\bbR$ be a real form of $\fg$ and let $\eta$ be the corresponding 
conjugation on $\fg$.
One can find a compact conjugation $\delta$ on $\fg$
satisfying $\eta\circ\delta=\delta\circ\eta:=\theta$ and 
the correspondence $\fg_\bbR\to\theta$ defines a bijection
\[\{\text{real forms $\fg_\bbR$ of}\ \fg\}/\text{isom}\longleftrightarrow\{\text{holomorphic involutions $\theta$ of }\fg\}/\on{Aut}(\fg)\]
between isomorphism classes of real forms and conjugacy classes of holomorphic involutions of 
$\fg$.
In this paper we show that 
there is continuous family of involutions
on $\fg$
interpolating the conjugation $\eta$ and the holomorphic involution $-\theta$
}
%%%%%%%%%%

\subsubsection{Real-symmetric homeomorphisms for matrices}
Let us first illustrate our main results with a notable case accessible to a general audience.

Let $\frak{gl}_n(\bbC)\is\bC^{n^2}$ denote the space of $n\times n$ complex matrices.
%A matrix $A\in M_n(\bbC)$ lies in $\calN_n(\bbC)$ if and only if any/all of the following equivalent conditions hold:
%% lies in $\calN_n(\bbC)$ if and only if  any of  the following  equivalent conditions hold:
%\begin{enumerate}
%
%\item The characteristic polynomial  $\chi_A(t) = \det(A-t I_n)$ is equal to $(-1)^n t^n$.
%\item $A$ is conjugate to a strictly upper triangular matrix.
%%\item $A^m=0$, for some $m\geq 1$.
%\item $A^n=0$.
%
%\end{enumerate}
Let $\frak{gl}_n(\bbR)\subset \frak{gl}_n(\bC)$ denote the real matrices,
i.e. those with real entries,
and
 $\fp_n(\bbC)\subset  \frak{gl}_n(\bbC)$ the  symmetric matrices,
 i.e. those equal to their transpose.
Introduce the following subspaces 
\[\frakgl_n'(\bbR)=\{x\in\frakgl_n(\bbR)|\text{\ eigenvalues of } x \text{ are real}\}\]
\[\frakp_n'(\bbC)=\{x\in\frakp_n(\bC)|\text{\ eigenvalues of } x \text{ are real}\}.\]
The real general linear group $\GL_n(\bbR)$ and complex orthogonal group $\Orth_n(\bbC)$ naturally act by conjugation on $\frak{gl}'_n(\bbR)$ 
 and $\fp_n'(\bC)$ respectively.
 The real orthogonal group $\Orth_n(\bbR) = \GL_n(\bbR) \cap \Orth_n(\bbC)$ 
acts on both $\frak{gl}_n'(\bbR)$
and $\fp_n'(\bC)$.
We also have the natural linear $\bbR^\times$-actions
 on both $\frak{gl}_n'(\bbR)$
and $\fp_n'(\bC)$.
 %and we have
%$\dim_\bbR \calN_n(\bbR) = n^2 -n$ and $\dim_\bbC \calN_n^{\sym}(\bbC) = (n^2 -n)/2$.
%The adjoint quotient map 
%$\chi$ restricts to maps from $\frak{gl}_n(\bbR),\fp_n(\bC)$ and $\ft_n(\bbR)$
%to $\bC^n$
%and we consider 
%the following fiber products
%\[\frak{gl}_n'(\bbR)=\frak{gl}_n(\bbR)\times_{\bC^n}\ft_n(\bbR)\ \ \ \ \ 
%\fp_n'(\bC)=\fp_n(\bC)\times_{\bC^n}\ft_n(\bbR).\]
 Consider the adjoint quotient map $\chi:\frak{gl}_{n}(\bC)\to\bC^n$
which
associates to each matrix $x\in\frak{gl}_n(\bC)$ the coefficients of its characteristic polynomial. Equivalently, one can think of it as giving the eigenvalues of the matrix (with
multiplicities).

Here is a notable  case of our general results.

\begin{thm}\label{intro: gln} 
There is an $\Orth_n(\bbR)\times\bbR^\times$-equivariant homeomorphism
\begin{equation}\label{eq:intro homeo gln}
\xymatrix{
\frak{gl}_n'(\bbR)\ar[r]^\sim&\fp'_n(\bbC)}
\end{equation}
which is compatible with the adjoint quotient map.
Furthermore, the homeomorphism restricts to a real analytic isomorphism between individual $\GL_n(\bbR)$-orbits and $\Orth_n(\bbC)$-orbits.
\end{thm}

%%%%%%%%%%%%%%%
\quash{

\begin{example}[n=2]

\begin{figure}[h!]
\centering
\hspace{2em}
\includegraphics[scale=0.5, trim={4.8cm 18.3cm 13.4cm 4.2cm},clip]{cone}
\hspace{8em}
\includegraphics[scale=0.45, trim={5cm 18.3cm 10cm 4cm},clip]{lines}
\begin{equation*}
\xymatrix{
\calN_2(\bbR) = \{ x^2 + y^2 = z^2\} \subset \bbR^3
&
 \calN^{\sym}_2(\bbC) =\{ uv = 0\} \subset \bbC^2
}
\end{equation*}
\caption{}\label{f: cone}
\end{figure}

When $n=2$, the real nilpotent cone $\calN_2(\bbR)$ is a real two-dimensional cone; the symmetric nilpotent cone $\calN^\sym_2$ is a nodal pair of complex lines (in Figure~\ref{f: cone}, we are only able to draw its real points). The theorem provides an $\Orth_2(\bbR)$-equivariant  homeomorphism $\calN_2(\bbR) \simeq   \calN^{\sym}_2(\bbC)$.

\end{example}

}
%%%%%%%%%%%%%

We deduce Theorem~\ref{intro: gln}  from the following more fundamental structure of linear algebra.
Consider the subspace 
\[\frak{gl}_n'(\bC)=\{x\in \frak{gl}_n(\bC)| \text{\ eigenvalues of } x \text{ are real}\}.\]
Let $\chi':\frakgl_n'(\bC)\to\bC^n$ be the restriction of the adjoint quotient map
to $\frak{gl}_n'(\bC)$.

\begin{thm}\label{intro: inv}
There is a continuous one-parameter family of $\Orth_n(\bbR)\times\bbR^\times$-equivariant maps 
\begin{equation}
\xymatrix{
\alpha_a:\frak{gl}_n'(\bC)\ar[r] & \frak{gl}_n'(\bC) & a\in [0,1]
}
\end{equation}
satisfying the properties: 
\begin{enumerate}
\item $\alpha_a^2$ is the identity, for all $a\in [0,1]$. 
\item $\alpha_a$ commutes with 
$\chi':\frakgl_n'(\bC)\to\bC^n$.
\item $\alpha_a$ takes each $\GL_n(\bbC)$-orbit real analytically to a $\GL_n(\bbC)$-orbit, for all $a\in [0,1]$.
\item At $a=0$, we recover conjugation: $\alpha_0(A) = \bar A$.
\item At $a=1$, we recover transpose: $\alpha_1(A) = A^{t}$
\end{enumerate}
\end{thm}

%Theorem~\ref{intro: gln} is a special case of a general result in Lie theory.

%, namely a lift of the Kostant-Sekiguchi correspondence
%(\cite{})
%from orbit posets to an equivariant homeomorphism.

\subsubsection{Real-symmetric homeomorphisms for Lie algebras}\label{intro: real-symmetric for Lie algebras}
To state a general version of our main results, we next recall some standard constructions in Lie theory, in particular
in the study of real reductive groups.

Let $G$ be a complex reductive Lie group with Lie algebra $\fg$.
Let $\fc=\fg//G$ be the categorical quotient
with respect to the adjoint action of $G$ on $\fg$. The adjoint quotient map
$\chi:\fg\to\fc$ is the \emph{Chevalley map}.

Let $G_\bbR\subset G$ be a real form, defined by a conjugation $\eta:G\to G$,
with Lie algebra $\fg_\bbR \subset \fg$. 
Choose a Cartan conjugation $\delta:G\to G$ that commutes with $\eta$, and let $G_c\subset G$ be the corresponding maximal compact subgroup.

 Introduce the Cartan involution $\theta = \delta \circ \eta:G\to G$, and let $K \subset G$ be the fixed subgroup of $\theta$ 
 with Lie algebra $\frakk \subset \fg$.
The subgroup $K$ is called the symmetric subgroup.
We have the Cartan decomposition 
 $\fg=\frakk \oplus \fp$ where 
$\fp\subset \fg$ is the $-1$-eigenspace of $\theta$.
Let $\fa\subset\fp$ be a maximal abelian subspace contained in $\fp$
and let $\ft\subset\fg$ be a $\theta$-stable Cartan subalgebra
containing $\fa$. 
Let $\rW_G=N_G(\ft)/Z_G(\ft)$ be the Weyl group 
of $G$
and $\rW=N_K(\fa)/Z_K(\fa)$ be the little Weyl group of the symmetric pair 
pair $(G,K)$.
We denote by $\fp_\bbR=\fp\cap\fg_\bbR$, $\frak k_\bbR=\frak k\cap\fg_\bbR$, 
$\fa_\bbR=\fa\cap\fg_\bbR$, etc.

One can organize the above groups into the diagram:
\beq\label{eq:diamond}
\xymatrix{&G&\\
K\ar[ru]\ar[ru]&G_c \ar[u]&G_\bbR\ar[lu]\\
&K_\bbR\ar[lu]\ar[u]\ar[ru]&}
\eeq
Here 
$K_\bbR$ is the fixed subgroup of $\theta, \delta$, and $\eta$ together (or any two of the three) and the maximal compact subgroup of $G_\mathbb R$ with complexification $K$.

Let $\fg_\bbR'\subset\fg_\bbR$ (resp. $\fp'\subset\fp$) be the subspace consisting of 
elements $x\in\fg_\bbR$ (resp. $x\in\fp$) such that the eigenvalues of the adjoint map $\ad_x:\fg\to\fg$ are real. 
The real from $G_\bbR$ and the symmetric subgroup $K$ act naturally on
$\fg_\bbR'$ and $\fp'$ by the adjoint action. The 
compact subgroup $K_\bbR=G_\bbR\cap K$ and $\bbR^\times$ act both on 
$\fg_\bbR'$ and $\fp'$.

\begin{thm}[Theorem \ref{homeomorphism for g}] \label{intro: main} 
Suppose $\fg$ is of classical type. %or $G_2$
There is a $K_\bbR\times\bbR^\times$-equivariant homeomorphism
\begin{equation}\label{eq: intro homeo}
\xymatrix{
\fg_\bbR' \ar[r]^\sim& \fp'
}
\end{equation}
which is compatible with the adjoint quotient map.
Furthermore, it restricts to a real analytic isomorphism between individual $G_\bbR$-orbits and $K$-orbits.
\end{thm}

%%%%%%%%%%%
\quash{
\begin{remark}
Our methods also provide interesting but incomplete results for the exceptional types $E_6, E_7, E_8, F_4, G_2$.
Namely, we obtain Kostant-Sekiguchi homeomorphisms for the 
minimal nilpotent orbits and the 
closed subcone $\calN_\sm \subset \calN $
of ``Reeder pieces" in the terminology of \cite{AH}. We conjecture it is possible to extend these homeomorphisms to all of $\calN$. 
See Remark~\ref{rem: exc}. 
\end{remark}
}
%%%%%%%%%%%

We deduce Theorem~\ref{intro: main}  from the following.
Let $\fc_{\fp,\bbR}\subset\fc$ be the image of the natural map
$\fa_\bbR\to\fc=\ft//\rW_G$.
Introduce $\fg'=\fg\times_{\fc}\fc_{\fp,\bbR}$ and let 
$\chi':\fg'\to\fc_{\fp,\bbR}$ be the projection map.

\begin{thm}[Theorem \ref{family of involutions for g}]\label{intro: main 2}
Under the same assumption as Theorem~\ref{intro: main},
there is a continuous one-parameter family of $K_\bbR\times\bbR^\times$-equivariant maps 
\begin{equation}\label{eq: intro family of invs}
\xymatrix{
\alpha_a:\fg'\ar[r] &  \fg' & a\in [0,1]
}
\end{equation}
satisfying the properties: 
\begin{enumerate}
\item $\alpha_a^2$ is the identity, for all $s\in [0,1]$. 
\item
$\alpha_a$ commutes with 
$\chi':\fg'\to\fc_{\fp,\bbR}$.
\item $\alpha_a$ takes each $G$-orbit real analytically to a $G$-orbit, for all $a\in [0,1]$.
\item At $a=0$, we recover the conjugation: $\alpha_0 = \eta$.
\item At $a=1$, we recover the anti-symmetry: $\alpha_1 = - \theta$.
\end{enumerate}
\end{thm}

\begin{remark}
The special case of Theorems~\ref{intro: main} and~\ref{intro: main 2} stated in Theorems~\ref{intro: gln} and \ref{intro: inv} is when $G=\GL_n(\bbC)$, $\fg \simeq\frakgl_n(\bbC)$, $G_\bbR= \GL_n(\bbR)$, $K= \Orth_n(\bbC)$, and $K_\bbR = \Orth_n(\bbR)$.
%In general, we first establish Theorems~\ref{intro: main} and~\ref{intro: main 2} for split type $A$, then check the constructions are compatible
%with inner automorphisms and Cartan involutions to deduce them for all classical types.
 \end{remark}

%%%%%%%%%%%%%%
\quash{
\begin{remark}
To deduce Theorem~\ref{intro: main} from Theorem~\ref{intro: main 2}, we consider the product $\fg' \times [0,1]$ with the horizontal vector field 
$\partial_s$ where $s$ is the coordinate on $[0,1]$. By averaging $\partial_s$ with respect to the $\bbZ/2\bbZ$-action given by $\alpha_s$, we obtain a new horizontal vector field on $\fg' \times [0,1]$. Its integral gives a new trivialization of $\fg' \times [0,1]$ taking $\fg'_\bbR \times\{0\}$ to $\fp'\times\{1\}$.  
%It would be interesting to compare this construction with Kronheimer's instanton flow~\cite{K}.
\end{remark}
}
%%%%%%%%%%%%%%

\subsubsection{Involutions on hyper-K\"ahler quotients
}
We deduce Theorem \ref{intro: main} and Theorem \ref{intro: main 2}
from a general result about 
involutions on
hyper-K$\dota$hler quotients of linear spaces.

Let $\mathbb H=\bbR\oplus\bbR i\oplus\bbR j\oplus\bbR k$ be the 
quaternions and let $\on{Sp}(1)\subset\mathbb H$
be the group consisting of elements of norm one.
Let $\textbf M$ be a 
be a finite dimensional quaternionic representation of a compact Lie group $H_u$.
We assume that the quaternionic representation is unitary, that is,
there is 
a $H_u$-inner product $(,)$ on $\textbf M$ which is hermitian with respect to
the complex structures $I,J,K$ on $\textbf M$ given by 
multiplication by $i,j,k$ respectively.
We have the hyper-K$\dota$hler moment map 
\[\mu:\textbf M\to \on{Im}\mathbb H\otimes\frak h_u^*\] vanishing at the origin.
Using the isomorphism $\on{Im}\mathbb H=\bbR\oplus\bC$ sending 
$x_1i+x_2j+x_3k$ to $(x_1,x_2+x_3i)$, we can identify
$\on{Im}\mathbb H\otimes\frak h_u^*=\frak h_u^*\oplus\frak h^*$ and 
hence obtain 
a decomposition of the moment map
\[\mu=\mu_\bbR\oplus\mu_\bC:\textbf M\to\frak h_u^*\oplus\frak h^*\]
of $\mu$ into real and complex components. 
We consider the hyper-K$\dota$hler quotient
\[\frak M_0=\mu^{-1}(0)/H_u\is\mu_\bC^{-1}(0)// H\]
where the right hand side is the categorial quotient
of $\mu_\bC^{-1}(0)$ by the complexification $H$ of $H_u$, and the second isomorphism 
follows from a result of Kempf-Ness \cite{KN}.

The hyper-K$\dota$hler quotient $\frak M_0$ has the following structures:
(1) it has a 
orbit type stratification 
\[\frak M_0=\bigsqcup_{(L)}\frak M_{0,(L)}\]
where a stratum $\frak M_{0,(L)}$ consists of orbits through 
points $x$ whose stabilizer in $H_u$ is conjugate to $L$,
(2) there is a hyper-K$\dota$hler 
$\on{SU}(2)=\on{Sp}(1)$-action on $\frak M_0$, denoted by $\phi(q):\frak M_0\to\frak M_0$, $q\in\on{Sp}(1)$, 
coming from the $\mathbb H$-module structure on $\textbf M$.

In Section \ref{inv on HK}, we prove the following general results
about involutions on  hyper-K$\dota$hler quotients.

\begin{thm}[Proposition \ref{family of involutions quiver} and Example \ref{norm function}]\label{intro: main 4}
\
\begin{enumerate}
\item
Let 
$\eta_H$ and $\eta_{\textbf M}$ be complex conjugations on 
$H$ and $\textbf M$ which are compatible with the unitary-quaternionic representation
of $H_u$ on $\textbf M$ (see Definition \ref{quaternionic} for the precise definition).
Then $\eta_H$ and $\eta_{\textbf M}$ induce
an anti-holomorphic involution 
\beq\label{intro: conjugation}
\eta:\frak M_0\lra\frak M_0
\eeq
such that 
the 
composition 
of $\eta$ with the 
hyper-K$\dota$hler $\on{SU}(2)$-action
of $q_a=\cos(\frac{a\pi}{2})i+\sin(\frac{a\pi}{2})k\in\on{Sp}(1)$
on $\frak M_0$, $a\in[0,1]$,
gives rise to a continuous family of involutions
\beq\label{intro: involutions}
\alpha_a:\frak M_0\lra\frak M_0\ \ \ \ a\in[0,1]
\eeq
interpolating the anti-holomorphic involution
$\alpha_0=\phi(i)\circ\eta$ and the holomorphic involution $\alpha_1=\phi(k)\circ\eta$.
\item
Let $\frak M_0(\bbR)$ and $\frak M_0^{sym}(\bC)$ be the fixed points 
of $\alpha_0$ and $\alpha_1$ on $\frak M_0$ respectively.
Then the intersection of the stratum $\frak M_{0,(L)}$ with 
$\frak M_0(\bbR)$ (resp. $\frak M_0^{sym}(\bC)$) define a stratification of 
$\frak M_0(\bbR)$ (resp. $\frak M_0^{sym}(\bC)$) and there exists a 
stratified homeomorphism
\beq\label{intro: homeo}
\xymatrix{\frak M_0(\bbR)\ar[r]^{\sim\ \ }&\frak M^{sym}_0(\bC)}
\eeq
which is real analytic on each stratum.  
\end{enumerate}
\end{thm}

\begin{remark}
Let $G, G_\bbR, G_u, K_\bbR$ be as in Section \ref{intro: real-symmetric for Lie algebras}.
Suppose that $\bfM$ is a unitary quaternionic representation of the larger 
group $H_u\times G_u$ and
the conjugations $\eta_H\times\eta_G$ and 
$\eta_{\bfM}$ on
$H\times G$ and $\bfM$
are compatible with the unitary-quaternionic representation.
Then the hyper-K$\dota$hler quotient $\frak M$ carries an action of 
$K_\bbR$ such that the involutions~\eqref{intro: conjugation},~\eqref{intro: involutions}, and homeomorphism~\eqref{intro: homeo} are 
$K_\bbR$-equivariant. 

\end{remark}

%Using the isomorphism $\on{Im}\mathbb H=\bbR\oplus\bC$ sending 
%$x_1i+x_2j+x_3k$ to $(x_1,x_2+x_3i)$, we can identify
%$\on{Im}\mathbb H\otimes\frak h_u^*=\frak h_u^*\oplus\frak h^*$ and 
%hence obtain 
%a decomposition of the moment map
%\[\mu=\mu_\bbR\oplus\mu_\bC:\textbf M\to\frak h_u^*\oplus\frak h^*\]
%of $\mu$ into real and complex components. 
%The map 
%$\mu_\bC:\textbf M\to\frak h^*$
%is holomorphic with respect to the complex structure 
%$I$ on $\textbf M$.
%, moreover, it is 
%$H$-equivariant with respect to the complex representation of $H$ on
%$\textbf M$ and the the inverse of the adjoint representation on $\frak h^*$.
%Kobak-Swann \cite{KS}

It is well-known that 
the complex nilpotent cone 
$\calN_n(\bC)\subset\frakgl_n(\bC)$ 
is an example of hyper-K$\dota$hler quotients
known as Nakajima's 
quiver varieties (see \cite{KP}, \cite{KS}, \cite{M},
\cite{Nak1}).
Applying Theorem \ref{intro: main 4} to this particular example, we obtain 
a family of $\on{O}_n(\bC)$-equivariant involutions 
\beq\label{intro: inv on nilcone}
\alpha_a:\calN_n(\bC)\lra\calN_n(\bC)\ \ \ \ a\in[0,1]
\eeq
interpolating the complex conjugation $\alpha_0(M)=\overline M$
and the transpose $\alpha_1(M)=M^t$, and a $\on{O}_n(\bC)$-equivariant
homeomorphism  
\beq\label{intro: homeo for nilcone}
\xymatrix{\calN_n(\bbR)\ar[r]^{\sim\ }&\calN_n^{sym}(\bC)}
\eeq
between real and symmetric nilpotent cone which restricts to a real analytic isomorphism between individual 
$\GL_n(\bbR)$-orbits and $\on{O}_n(\bC)$-orbits.
This establishes a special case of 
Theorems \ref{intro: main} and \ref{intro: main 2} for
the  
fiber of the adjoint quotient map 
$\chi':\frakgl'_n(\bC)\to\bC^n$ over $0\in\bC^n$, that is,
matrices with zero eigenvalues.
To extend the results 
to 
matrices with real eigenvalues, 
we prove a version 
of Theorem \ref{intro: main 4} for the family of 
hyper-K$\dota$hler quotients 
\[\frak M_{Z_\bC}=\mu^{-1}_\bbR(0)\cap\mu^{-1}_\bbC(Z_\bC)/H_u\to Z_\bC\]
where $Z_\bC\subset\frak h^*$ is the dual of the center of $\frak h$ and 
then 
deduce the results using the description of 
general adjoint orbits closures as quiver varieties 
in \cite{MV}.
Finally, we check that the constructions are compatible
with inner automorphisms and Cartan involutions and then 
deduce the case of Lie algebras of classical types from the case of $\frakgl_n(\bC)$.

We would like to emphasize that the keys 
in the proof of Theorems \ref{intro: main} and \ref{intro: main 2}
are 
the symmetries on adjoint orbit closures (or rather, the symmetries on the whole family 
$\frakgl_n'(\bC)\to\bC^n$)
coming from the 
hyper-K$\dota$hler $\on{SU}(2)$-action.
Those symmetries are not immediately visible in their original definitions
as algebraic varieties.

\begin{remark}
The use of hyper-K$\dota$hler $\on{SU}(2)$-actions in the study of 
geometry of 
nilpotent orbits goes back to the 
celebrated work of Kronheimer \cite{Kr} where 
he used those symmetries to 
gave a differential-geometric interpretation of 
Brieskorn's theorem on sub-regular singularities.

\end{remark}

\subsection{Applications }
We discuss here applications to the Kostant-Sekiguchi correspondence,
singularities of real and symmetric 
adjoint orbit closures,
and Springer theory for symmetric spaces.

In the rest of the section, we assume 
$\fg$ is of classical type.

\subsubsection{Generalized Kostant-Sekiguchi homeomorphisms}
The celebrated Kostant-Sekiguchi correspondence is an isomorphism 
between
real and symmetric nilpotent orbit 
posets
\beq
|G_\bbR\bs\calN_{\bbR}|\longleftrightarrow|K\bs\calN_{\fp}|.
\eeq
The bijection was proved by Kostant (unpublished) and Sekiguchi~\cite{S}. Vergne~\cite{V}, using
Kronheimer's instanton flow~\cite{Kr},  showed the corresponding orbits are diffeomorphic. 
Schmid-Vilonen ~\cite{SV} gave an alternative proof and further refinements using Ness'  moment map.
Barbasch-Sepanski~\cite{BaSe} deduced the bijection is a poset isomorphism  from Vergne's results.

We shall state a lift/generalization of the Kostant-Sekiguchi correspondence to stratified homeomorphisms between adjoint orbits closures in the real Lie algebra $\fg_\bbR$ and symmetric
subspace $\fp$ whose eigenvalues are real but not necessarily zero.

Denote by 
$\calN_\xi=\chi^{-1}(\xi)$ the fiber of the Chevalley map 
$\chi:\fg\to\fc$ over 
$\xi\in\fc$.
In \cite{Ko1}, Kostant proved that 
there are finitely many $G$-orbits in $\calN_\xi$ and there is a 
unique closed orbit $\mO_\xi^{\fraks}$ consisting of semisimple elements and 
a unique open orbit $\mO_{\xi}^{\frakr}$ consisting of regular elements.
Moreover, we have 
$\calN_\xi=\overline{\mO_{\xi}^{\frakr}}$.

Assume $\xi\in\fc_{\fp,\bbR}\subset\fc$. 
Then $\xi$ is fixed by the involutions on $\fc$
induced by 
$\eta$ and $-\theta$ and hence the fiber $\calN_\xi$ is stable under $\eta$ and $-\theta$.
We write 
\[\calN_{\xi,\bbR}=\calN_{\xi}\cap \fg_\bbR\ \ \ \ \ \calN_{\xi,\fp}=\calN_\xi\cap\fp.\]
for the fixed points.
There are finitely many 
$G_\bbR$-orbits and $K$-orbits on $\calN_{\xi,\bbR}$ and $\calN_{\xi,\fp}$
\[\calN_{\xi,\bbR}=\bigsqcup_l\mO_{\bbR,l}\ \ \ \ \calN_{\xi,\fp}=\bigsqcup_l\mO_{\fp,l}\]

\begin{corollary}\label{intro: main 3}
There is a $K_\bbR$-equivaraint stratified homeomorphism
\beq\label{generalized KS homeo}
\xymatrix{\calN_{\xi,\bbR}\ar[r]^{\sim}&\calN_{\xi,\fp}}
\eeq
which restricts to real analytic isomorphisms between individual $G_\bbR$-orbits and $K$-orbits.
The homeomorphism induces an isomorphism between $G_\bbR$-orbits and $K$-orbits posets
\beq\label{generalized KS cor}
|G_\bbR\bs\calN_{\xi,\bbR}|\longleftrightarrow|K\bs\calN_{\xi,\fp}|.
\eeq
\end{corollary}
\begin{proof}
It follows immediately from Theorem \ref{intro: main}.
\end{proof}

\begin{remark}
Thanks to the work of Vergne~\cite{V},
it is known that under the Kostant-Sekiguchi bijection the correspondence orbits are 
diffeomorphic. It is an open question whether the corresponding orbit closures have
the same singularities (see, e.g., \cite[Introduction]{He}).  
Corollary \ref{intro: main 3} gives a positive answer in the case of classical Lie algebras.
\end{remark}

\begin{remark}
In \cite{Bie} and \cite{Biq}, the authors proved an extended Kostant-Sekiguchi correspondence for 
certain adjoint orbits. We expect that their correspondence is compatible with the one 
in~\eqref{generalized KS cor}.
\end{remark}

\begin{remark}
In Theorem \ref{KS for quivers},
we 
also establish a 
Kostant-Sekiguchi correspondence between real and symmetric symplectic leaves 
for quiver varieties.
\end{remark}

\subsubsection{Derived categories}

Let $D_{G_\bbR}(\calN_{\xi,\bbR})$, $D_K(\calN_{\xi,\fp})$ denote the respective equivariant derived categories of sheaves (over any commutative ring).
Since $K_\bbR \to G_\bbR$, $K_\bbR\to K$ are homotopy equivalences, the forgetful functors $D_{G_\bbR}(\calN_\bbR) \to D_{K_\bbR}(\calN_\bbR)$, $D_K(\calN_\fp)\to D_{K_\bbR}(\calN_\fp)$ to $K_\bbR$-equivariant complexes are fully faithful with essential image those complexes constructible along the respective orbits of $G_\bbR$ and $K$.

Transport along the homeomorphism of Theorem \ref{intro: main 3} immediately provides:

\begin{corollary}
Pushforward along the homeomorphism \eqref{generalized KS homeo} provides an equivalence of equivariant derived categories
\begin{equation}\label{eq: intro equiv}
\xymatrix{
D_{G_\bbR}(\calN_{\xi,\bbR}) \simeq D_K(\calN_{\xi,\fp})
}
\end{equation}
\end{corollary}

\subsubsection{Vanishing of odd dimensional intersection cohomology}
Theorem \ref{intro: main 3} implies that the singularities of symmetric nilpotent orbit closures 
$\bar\mO_\fp \subset \calN_\fp$
are homeomorphic  to the singularities of the corresponding real nilpotent
orbit closures $\bar\mO_\bbR\subset \calN_\bbR$.
Thus we can deduce results about one from the other.

Here is a notable example.
Let $\on{IC}(\mO_\bbR,\mL_\bbR)$ be the intersection cohomology sheaf
of a real nilpotent orbit $\mO_\bbR \subset \calN_\bbR$ with coefficients in
a $G_\bbR$-equivariant local system $\mL_\bbR$. (Recall that 
all nilpotent orbits $\mO \subset \calN$ have even complex dimension, so all real nilpotent orbits $\mO_\bbR \subset \calN_\bbR$ have even real dimension, hence 
middle perversity makes sense.)
%We use the normalization so that 
%$\on{IC}(\mO_\bbR,\mL)[\dim\mO_\bbR/2]$ is a perverse sheaf with respect to the middle perverse $t$-structure.

\begin{corollary}
The cohomology sheaves 
$\calH^i(\on{IC}(\mO_\bbR,\mL_\bbR))$ vanish for $i-\dim_\bbR \calO_\bbR /2$ odd.
\end{corollary}
\begin{proof}
Using the equivalence~\eqref{eq: intro equiv}, 
it suffices to prove the asserted vanishing for
 the intersection cohomology sheaf $\on{IC}(\mO_\fp,\mL_\fp)$ 
of a symmetric nilpotent orbit $\mO_\fp \subset \calN_\fp$ with coefficients in
a $K$-equivariant local system $\mL_\fp$, and
$i - \dim_\bbC \calO_\fp$ odd.
This is proved in 
\cite[Theorem 14.10]{LY}.\footnote{In fact, 
\cite{LY} establishes the odd vanishing in the 
more general setting of graded Lie algebras.}
 \end{proof}

\begin{remark}
The proof of \cite[Theorem 14.10]{LY} makes use of Deligne's theory of weights and the theory of canonical bases,
and hence does not have an evident generalization to a real algebraic setting.

\end{remark}

\subsubsection{Formula for the sheaf of symmetric nearby cycles}
Consider the quotient map
$\chi_\fp:\fp\to\fc_{\fp}=\fp//K$.
According to \cite{KR}, the generic fiber of 
$\chi_\fp$ is a single $K$-orbit through a 
semisimple element in $\fp$ and the special fiber 
over the base point 
$\chi_\fp(0)\in\fc_\fp$ is the symmetric nilpotent cone 
$\calN_\fp$.
Following Grinberg \cite{G1} (see also \cite{G2,GVX}), 
we consider the sheaf 
$\mF_\fp\in D_K(\calN_\fp)$ of nearby cycles along the special fiber
$\calN_\fp$ in the family $\chi_\fp:\fp\to\fc_{\fp}$ (see Section \ref{symmetric NC} for the precise definition).
We will call $\mF_\fp$ the \emph{sheaf of symmetric nearby cycles}.

Let $B_\bbR\subset G_\bbR$ be a minimal parabolic subgroup with Lie algebra
$\fb_\bbR=\frak m_\bbR+\fa_\bbR+\frak n_\bbR$ where 
$\frak m_\bbR=Z_{\frak k_\bbR}(\fa_\bbR)$ and $\frak n_\bbR$ is the nilpotent radical.
Consider the real Springer map 
\[\pi_\bbR:\widetilde\calN_\bbR\to\calN_\bbR\]
where $\widetilde\calN_\bbR=G_\bbR\times^{B_\bbR}\frak n_\bbR$
and $\pi_\bbR(g,v)=\on{Ad}_gv$.

We have the following formula for the sheaf of symmetric nearby cycles:

\begin{corollary}[Theorem \ref{S=R}]\label{formula for nearby cycles}
Under the equivalence 
$D_K(\calN_\fp)\is D_{G_\bbR}(\calN_\bbR)$~\eqref{eq: intro equiv}, the sheaf of 
symmetric nearby cycles $\mF_\fp$ becomes the real Springer sheaf 
$\mS_\bbR:=(\pi_\bbR)_!\bC[\dim_\bbR\calN_\bbR/2]$.
In particular,
the real Springer map 
$\pi_\bbR:\widetilde\calN_\bbR\to\calN_\bbR$ is a semi-small map and the 
real Springer sheaf $\calS_\bbR$
is a perverse sheaf.
\end{corollary}

In fact, Theorem \ref{S=R} is slightly stronger than the one stated here.
We also prove a formula for the sheaf of symmetric nearby cycles with coefficients in 
$K$-equivariant local systems and we show that,
for 
 any $\fg_\bbR$ (not just for classical types),
the real Springer sheaf is isomorphic to the \emph{sheaf of
real nearby cycles} $\mF_\bbR$ 
introduced in Section \ref{Real NC}.

\begin{remark}
The formula above for  
symmetric nearby cycles
was originally conjectured by 
Vilonen-Xue and the first author.
It 
can be viewed as 
a symmetric space version of the 
well-known result that the sheaf of 
nearby cycles along the 
special fiber $\calN$ in the family 
$\chi:\fg\to\fc$ is isomorphic to the Springer sheaf.

\end{remark}

\begin{remark}
In \cite{CVX}, the authors used the sheaves of symmetric nearby cycles (with coefficients)
to produce all cuspidal complexes on $\calN_\fp$ and use them to establish a
Springer correspondence for the split symmetric pair of type A (see \cite{VX} for the cases of classical symmetric pairs). The 
formula established in Corollary \ref{formula for nearby cycles}
provides new
insights and methods into the study of Springer theory for general symmetric pairs
and real groups.
We will give one example below. The details will 
be discussed in a sequel \cite{CN2}.
\end{remark}

\subsubsection{Real Springer theory and Hecke algebras at roots of unity}
In \cite{G1}, Grinberg gave a generalization of 
Springer theory using nearby cycles. One of the main results in \emph{loc. cit.} is a description of the 
endomorphism algebra $\End(\mF_\fp)$ of the sheaf of symmetric nearby cycles 
as
a certain Hecke algebra at roots of unity.\footnote{In fact, he works in a more general setting of polar representations.}
To explain his result, let $(\Phi,\fa_\bbR^*)$ be the root system (possible non-reduced)
of $(\fg_\bbR,\fa_\bbR)$. 
For each $\alpha\in\Phi$
we denote by $\fg_{\bbR,\alpha}\subset\fg_\bbR$ the corresponding $\alpha$-eigenspace. 
Choose a system of simple roots $\Delta\subset\Phi$ and let $\mathrm S\subset\rW$ be the set of simple reflections 
of the little Weyl group associated to $\Delta$.
Consider the following algebra
\[\mathcal H_{G_\bbR}:=\bC[B_\rW]/(T_{s}-1)(T_{s}+(-1)^{d_{s}})_{s\in\mathrm S}\]
where 
$\bC[B_\rW]$ is the group algebra of the braid group $B_\rW$
of $\rW$ with generators $T_s$, $s\in\mathrm S$,
and $d_s$ is the integer given by
\[d_s=\sum_{\alpha\in\Delta, s_\alpha=s}\dim_{\mathbb R}(\fg_{\bbR,\alpha}),\] 
where $s_\alpha$ denotes the reflection
corresponding to the simple root $\alpha\in\Delta$\footnote{Since the root system might not be reduced, there 
might be more than one simple root $\alpha$ such that $s_\alpha=s$.}.
For examples, if $G_\bbR$ is a split real from, then 
we have 
$d_s=1$ for all $s\in\mathrm S$ and  
$\mathcal H_{G_\bbR}$ is isomorphic to the Hecke algebra associated to 
$\rW$ at $q=-1$.
On the other hand,
if $G_\bbR$ is a complex group, then
we have $d_s=2$ and
$\mathcal H_{G_\bbR}$ is isomorphic to the 
group algebra $\bC[\rW]$.

In \cite[Theorem 6.1]{G1}, Grinberg showed that there is a canonical isomorphism 
of algebras
\beq\label{Hecke algebra}
\End(\mF_\fp)\is\mathcal H_{G_\bbR}.
\eeq
Since 
the algebra $\mathcal H_{G_\bbR}$ is in general not semi-simple, 
as an interesting corollary of~\eqref{Hecke algebra}, we see that 
the sheaf of symmetric nearby cycles $\mF_\fp$
is not semi-simple in general.
 
Now combining Corollary \ref{formula for nearby cycles} with 
Grinberg's theorem, we obtain the 
the following result in real Springer theory:

\begin{corollary}\label{Real Springer}
 
We have a canonical isomorphism of 
algebras
\[\End(\mS_\bbR)\is\mathcal H_{G_\bbR}.\]
In particular, the
real Springer sheaf $\calS_\bbR$ is in general not semi-simple and,
for any 
$x\in\calN_\bbR$, the cohomologies 
$\on{H}^*(\mathcal{B}_{x},\bC)$ of the 
real Springer fiber 
$\calB_{x}=\pi_\bbR^{-1}(x)$ carry a natural action
of the algebra $\mathcal H_{G_\bbR}$.

\end{corollary}

\begin{remark}
In \cite{CN2}, we will give an alternative proof of 
Corollary \ref{Real Springer} (for all types) following the classical arguments in Springer theory.
In particular, combining with Corollary \ref{formula for nearby cycles}, we obtain a new proof of 
Grinberg's theorem on the endomorphism algebra of $\mF_\fp$.
\end{remark}

\subsection{Previous work}
In our previous work \cite{CN1}, we 
establish Corollary \ref{intro: main 3} 
in the case of nilpotent cone
using 
the 
geometry of 
moduli space of
quasi-maps associated to a symmetric pair $(G,K)$.
%\footnote{In \emph{loc. cit.} we only deal with the case 
%of nilpotent cones. The generalization to all matrices with real eigenvalues will appear in a revised version of the paper.}.
In more detail, 
we use  the factorization properties of the moduli space of
quasi-maps to establish a real-symmetric homeomorphism 
in the setting of Beilinson-Drinfeld Grassmannians (for any reductive group $G$) and then deduce 
Corollary \ref{intro: main 3} using the
Lusztig embedding of the nilpotent cone for $\frakgl_n(\bC)$
 into
the affine Grassmannian for $\GL_n(\bC)$. 
The result in the present paper suggests that there should be a 
hyper-K$\dota$hler geometry interpretation of the 
results in \cite{CN1}.
This will be discussed in detail in a sequel.

We conclude the introduction with the following conjecture.
\begin{conjecture}
Theorems \ref{intro: main} and \ref{intro: main 2}
remain true when 
$\fg$ is of exceptional type.
\end{conjecture}

\subsection{Organization}

We briefly summarize here the main goals of each section. 
In Sect.~\ref{inv on HK} immediately to follow, we 
study involutions on hyper-K$\ddot{\text a}$hler quotients
of linear spaces.
In Sect.~\ref{quiver varieties}, we apply the results established in the previous section 
to the case of quiver varieties.
In Sect.~\ref{main results}, we establish our main results 
Theorems \ref{homeomorphism for g} and \ref{family of involutions for g}. 
In Sect.~\ref{applications}, we discuss applications to 
Springer theory for real groups and symmetric spaces.

\subsection{Acknowledgements} 
The authors would like to thank Marco Gualtieri for inspiring discussions about symmetries of hyper-K$\ddot{\text a}$hler quotients and Jeffrey Adams for useful discussions about 
Kostant-Sekiguchi correspondence.
We also would like to thank Kari Vilonen and Ting Xue for useful comments.

The research of
T.H. Chen is supported by NSF grant DMS-1702337
and that of D. Nadler  by NSF grant DMS-1802373.

%%%%%%%%%%%%%%%%
%%%%%%%%%%%%%%%%%

\section{A family of involutions on hyper-K$\ddot{\text a}$hler quotients}\label{inv on HK}
In this section we introduce a family of involutions 
on hyper-K$\dota$hler quotients of linear spaces with remarkable properties.
The main references for hyper-K$\dota$hler quotients are \cite{Hi} and \cite{HKLR}.

\subsection{Quaternions}
Let $\mathbb H=\bbR\oplus\bbR i\oplus\bbR j\oplus\bbR k$ be the 
quaternions.
For any $x=x_0+x_1i+x_2j+x_3k\in\mathbb H$ we denote by
$\bar x=x_0-x_1i-x_2j-x_3k$. Then 
the paring 
$(x,x')=\on{Re}(x\bar x')$ defines a real-valued inner product on $\mathbb H$.
We denote by 
$\on{Im}(\mathbb H)=\bbR i\oplus\bbR j\oplus\bbR k$, the pure imaginary quaternions,
and $\on{Sp}(1)=\{x\in\mathbb H|(x,x)=1\}$ the group of quaternions of norm one.

%For any finite dimensional $\mathbb H$-vector space $V$, 
%we denote by $\GL_{\mathbb H}(V)$ (resp. $\GL_\bbC(V)$) the group of invertible $\mathbb H$-linear (resp. $\bC$-linear) endomorphisms of $V$.

\subsection{Hyper-K$\dota$hler quotient of linear spaces}
Let $H$ be a complex reductive group with
compact real form $H_u$.
Let $\textbf M$ 
be a quaternionic representation of $H_u$, that is, $\textbf M$ is 
a finite dimensional quaternionic vector space 
together with a $\mathbb H$-linear action of $H_u$.
We assume that the quaternionic representation is unitary, that is,
there is 
a $H_u$-inner product $(,)$ on $\textbf M$ which is hermitian with respect to
the complex structures $I,J,K$ on $\textbf M$ given by 
multiplication by $i,j,k$ respectively.
We have a natural complex representation of 
$H$ on $\textbf M$ preserving the 
complex symplectic form 
$\omega_\bC(v,v')=(Jv,v')+i(Kv,v')$ on 
$\textbf M$.

We have the hyper-K$\dota$hler moment map 
\[\mu:\textbf M\to \on{Im}\mathbb H\otimes_\bbR\frak h_u^*\] 
satisfying
\[\langle\xi,\mu(\phi)\rangle=(I\xi\phi,\phi)i+(J\xi\phi,\phi)j+(K\xi\phi,\phi)k\in\on{Im}\mathbb H\]
where $\xi\in\frak h_u$, $\phi\in\textbf M$, and $\langle,\rangle$ is the paring 
between 
$\frak h^*_u$ and $\frak h_u$.
The map $\mu$ has the following equivariant properties:
(1) it intertwines the 
$\on{Sp}(1)\times H_u$ action on $\textbf M$ 
and the one on 
$\on{Im}(\mathbb H)\otimes_\bbR\frak h_u^*$ given 
by $(q,h)(w,u)=(\on{Ad}_{q}w,\on{Ad}_{h^{-1}}u)$
(2) we have 
$\mu(tv)=t^2\mu(v)$ for $t\in\bbR^\times, v\in\textbf M$.

Using the isomorphism $\on{Im}\mathbb H=\bbR\oplus\bC$ sending 
$x_1i+x_2j+x_3k$ to $(x_1,x_2+x_3i)$, we can identify
$\on{Im}\mathbb H\otimes\frak h_u^*=\frak h_u^*\oplus\frak h^*$ and 
hence obtain 
a decomposition of the moment map
\[\mu=\mu_\bbR\oplus\mu_\bC:\textbf M\to\frak h_u^*\oplus\frak h^*\]
of $\mu$ into real and complex components. 
The map 
$\mu_\bC:\textbf M\to\frak h^*$
is holomorphic with respect to the complex structure 
$I$ on $\textbf M$ and satisfies
\[\langle\xi,\mu_\bC(\phi)\rangle=\omega_\bC(\xi\phi,\phi)\]
where $\xi\in\frak h$ and $\phi\in\textbf M$.
Moreover, it is 
$H$-equivariant with respect to the complex representation of $H$ on
$\textbf M$ and the the inverse of the adjoint representation on $\frak h^*$.

Let $Z=\{v\in\frak h_u^*|\on{Ad}_h(v)=v\text{\ for all\ }h\in H_u\}$ and 
$Z_\bC=\mathbb C\otimes_\bbR Z$. Then we have $\on{Im}\mathbb H\otimes_\bbR Z=Z\oplus Z_\bC$.
For any
$\zeta_\bC\in Z_
\bC$, we can consider 
the
hyper-K$\dota$hler quotient
\beq\mathfrak M_{\zeta_\bC}=
\mu^{-1}_\bbR(0)\cap\mu^{-1}_\bbC(-\zeta_\bbC)/H_u.
\eeq
We have the holomorphic description 
\[\mathfrak M_{\zeta_\bC}\is\mu_\bC^{-1}(-\zeta_\bC)// H\]
where the right hand side is the categorial quotient
of $\mu_\bC^{-1}(-\zeta_\bC)$ by $H$.
One can form a perturbed hyper-K$\dota$hler quotient
\[\mathfrak M_{(\zeta_\bbR,\zeta_\bC)}=\mu^{-1}_\bbR(-\zeta_\bbR)\cap\mu^{-1}_\bbC(-\zeta_\bbC)/H_u\]
with not necessarily zero real component $\zeta_\bbR$.
The composition  
$\mu^{-1}_\bbR(-\zeta_\bbR)\cap\mu^{-1}_\bbC(-\zeta_\bbC)\to\mu^{-1}_\bbC(-\zeta_\bbC)
\to\mu^{-1}_\bbC(-\zeta_\bbC)//H$
gives rise to a map
\beq\label{resolution}
\pi:\mathfrak M_{(\zeta_\bbR,\zeta_\bC)}\to\mathfrak M_{\zeta_\bC}
\eeq
which is holomorphic with respect to the complex structure $I$.

From now on we will fix a real parameter 
$\zeta_\bbR$. 
For any
subset 
$S\subset
Z_\bC$ we can consider the following family of hyper-K$\dota$hler quotients
\[\chi_S:\frak M_S=\mu_\bbR^{-1}(0)\cap\mu_\bbC^{-1}(-S)/ H_u\to S\] 
\[\tilde\chi_S:\frak M_{(\zeta_\bbR,S)}=\mu_\bbR^{-1}(-\zeta_\bbR)\cap\mu_\bbC^{-1}(-S)/H_u\to S\]
Then the map~\eqref{resolution} gives rise to a map 
\beq\label{resolution quiver}
\pi_S:\frak M_{(\zeta_\bbR,S)}\to\frak M_S
\eeq
compatible with the projection maps to $S$.

\subsection{A stratification}
Let $\zeta_\bC\in Z_\bC$.
Let  
$L$ be a subgroup of $H_u$. We denote by 
$\textbf M_{(L)}$ be the set of all points in $\textbf M$
whose stabilizer is conjugate to $L$.
A point in $\frak M_{\zeta_\bC}$ is said to be of stabilizer type $(L)$ if 
it has a representative in $\textbf M_{(L)}$.
The set of all points of stabilizer type $(L)$ is denoted by
$\frak M_{\zeta_\bC,(L)}$.
We have a \emph{orbit type} stratification
\beq\label{stra}
\frak M_{\zeta_\bC}=
\bigsqcup_{(L)}
\frak M_{\zeta_\bC,(L)}
\eeq
where the union runs over the set of all 
conjugacy classes of subgroups of $U$.
%\footnote{$\frak M_{\zeta_\bC,(L)}$
%is nonempty only if the complexification $L_\bC$ is a reductive subgroup 
%of $H$ and there are only finitely many 
%conjugacy classes of reductive subgroups of $H$.}.
Each stratum $\frak M_{\zeta_\bC,(L)}$ is a smooth hyper-K$\dota$hler manifold, moreover, it is an affine symplectic variety with respect to the complex structure $I$.
%and its symplectic
%leaves are precisely
%the strata $\frak M_{\zeta,(H)}$.

\subsection{Symmetries of hyper-K$\dota$hler quotients}
Let $G$ be another complex reductive group with 
a compact real form $G_u$.
Consider a unitary representation of 
$G_u$ on $\textbf M$
commuting with the $H_u$-action on $\textbf M$.
Then for any $S\subset Z_\bC$, the action 
of the complexification $G$ on $\bfM$ descends to an action on the hyper-K$\dota$hler quotient
$\mathfrak M_S$ which
is 
compatible with the projection map to $S$ and holomorphic with respect to the complex structure $I$.

Assume $S\subset Z_\bC$ is a $\bbR$-linear subspace.
Then the action of $\bbR^\times$ on 
$\textbf M$ descends to a $G$-equivaraint $\bbR^\times$-action 
on $\frak M_S$:
\beq\label{G_m action}
\phi(t):\frak M_S\to\frak M_S, \ \ 
\ t\in\bbR^\times.
\eeq
Moreover, we have a commutative diagram
\[\xymatrix{\frak M_S\ar[r]^{\phi(t)}\ar[d]^{}&\frak M_S\ar[d]^{}\\
S\ar[r]^{t^2(-)}&S}\] 
where the bottom arrow is the multiplication by $t^2$.

Let $q\in \on{Sp}(1)$ 
such that $\on{Ad}_q(S)\subset S$.
Here $\on{Ad}_q:\on{Im}\mathbb H\otimes Z\to \on{Im}\mathbb H\otimes Z$ is the map 
$\on{Ad}_q(w,u)=(\on{Ad}_qw,u)$ and
we identify $Z_\bC=(\bbR j\oplus\bbR k)\otimes Z$, and hence $S$, as a subspace 
of $\on{Im}\mathbb H\otimes Z$ with zero $i$-component.
The action of $q\in\on{Sp}(1)$   
on $\textbf M$ gives rise to a $G_u$-equivaraint map
\beq\label{r_q}
\xymatrix{\phi(q):\frak M_S\to\frak M_S}
\eeq
commuting with the $\bbR^\times$-actions.
In addition, we have the following commutative diagram
\[\xymatrix{\frak M_S\ar[r]^{\phi(q)}\ar[d]^{}&\frak M_S\ar[d]^{}\\
S\ar[r]^{}&S}\]
where the bottom arrow is 
$\on{Ad}_q:S\to S$.

It is straightforward to check that 
the stratum $\frak M_{\zeta_\bC,(L)}$ in~\eqref{stra} is stable under the 
$G$ and $\bbR^\times$-actions. 
Moreover, for any $q\in\on{Sp}(1)$ (resp. $t\in\bbR^\times$) and $S$ as above 
the map $\phi(q)$ (resp. $\phi(t)$) is compatible with the 
stratifications in the sense that it maps the 
stratum 
$\frak M_{\zeta_\bC,(L)}$ in the fiber 
$\chi_S^{-1}(\zeta_\bC)=\frak M_{\zeta_\bC}$
to the corresponding stratum 
$\frak M_{\zeta'_\bC,(L)}$ in the fiber 
$\chi_S^{-1}(\zeta'_\bC)=\frak M_{\zeta'_\bC}$
where $\zeta_\bbC'=\on{Ad}_q\zeta_\bC$ (resp. $\zeta_\bC'=t^2\zeta_\bC$).

\begin{example}
Let $S=0$. Then we have 
$\on{Ad}_q(0)=0$ for all $q\in\on{Sp}(1)$ and the family of maps $\phi(q)$ in~\eqref{r_q}
gives rise to a $(G_u\times\bbR^\times)$-equivarint 
$\on{Sp}(1)$-action on 
$\frak M_0$, called the \emph{hyper-K$\dota$hler $\on{Sp}(1)$-action}. Moreover the stratum 
$\frak M_{0,(L)}$ is stable under the $\on{Sp}(1)$-action.

\end{example}

\subsection{Conjugations on $\mathfrak M_{Z_\bC}$}\label{family of conjugations}

\begin{definition}\label{symplectic}
Let $\eta_H$ and $\eta_{\textbf M}$ be conjugations on 
$H$ and $\textbf M$ respectively.
We say that $\eta_H$ and $\eta_{\textbf M}$ are compatible with the 
symplectic representation of $H$ on $\textbf M$
 if the following holds:
\begin{enumerate}
\item
we have 
$\eta_{\textbf M}(hv)=\eta_H(h)\eta_{\textbf M}(v)$
for all $h\in H$ and $v\in\textbf M$.

\item
we have 
$\omega_\bC(\eta_{\textbf M}(v),\eta_{\textbf M}(v'))=\overline{\omega_\bC(v,v')}$
for all $v,v'\in\textbf M$.

\end{enumerate}

 \end{definition}

\begin{lemma}\label{compatibility with mu_C}
Let $\eta_H$ and $\eta_{\textbf M}$ be conjugations on 
$H$ and $\textbf M$ compatible with the 
symplectic representation of $H$ on $\textbf M$.
Then the
complex moment map
$\mu_\bC:\textbf M\to\frak h^*$ 
intertwines 
$\eta_{\textbf M}$ and $\eta_H$.
\end{lemma}
\begin{proof}
For any $\xi\in\frak h, v\in\textbf M$, we have 
\[\langle\xi,\mu_\bC(\eta_{\textbf M}(v))\rangle=\omega_\bC(\xi\eta_{\textbf M}(v),\eta_{\textbf M}(v))=
\omega_\bC(\eta_{\textbf M}(\eta_H(\xi)v),\eta_{\textbf M}(v))=
\overline{\omega_\bC(\eta_H(\xi)v,v)}=\]
\[=\overline{\langle\eta_H(\xi),\mu_\bC(v)\rangle}=\langle\xi,\eta_H(\mu_\bC(v))\rangle.\]
This implies $\mu_\bC(\eta_{\textbf M}(v))=\eta_H(\mu_\bC(v))$
for all $v\in\textbf M$. The lemma follows.
\end{proof}

Let $\eta_H$ and $\eta_{\textbf M}$ be as in Lemma \ref{compatibility with mu_C}.
Then the center of $\frak h$, and hence 
$Z_\bC$, is stable under 
$\eta_H$. 
It follows that, for any 
$\zeta_\bC\in Z_\bC$, the conjugation $\eta_{\textbf M}$ on $\textbf M$ descends to a map
\beq\label{fiber bar}
\frak M_{\zeta_\bC}=\mu_\bC^{-1}(-\zeta_\bC)//H\to\frak M_{\eta_{H}(\zeta_\bC)}=\mu_\bC^{-1}(-\eta_{H}(\zeta_\bC))//H
\eeq
which is anti-holomorphic with respect to the complex structure $I$.
Moreover, it
maps the stratum $\frak M_{\zeta_\bC,(L)}$
to the corresponding stratum 
$\frak M_{\zeta_\bC,(L)}$.
As $\zeta_\bC$ varies over $Z_\bbC$, the map~\eqref{fiber bar}
 organize into a map
\beq\label{conjugation}
\eta_{Z_\bC}:\frak M_{Z_\bC}\to\frak M_{Z_\bC}
\eeq
making the following diagram commute
\beq\label{proj to Z_C}
\xymatrix{\frak M_{Z_\bC}\ar[r]^{\eta_{Z_\bC}}\ar[d]^{}&\frak M_{Z_\bC}\ar[d]^{}\\
Z_\bC\ar[r]^{\eta_H}&Z_\bC}.
\eeq
We will call $\eta_{Z_\bC}$ the \emph{conjugation}
on $\frak M_{Z_\bC}$ associated to the conjugations $\eta_H$ and $\eta_{\textbf M}$.

%Assume there is another unitary quaternionic representation 
%of $G$ on $\textbf M$.
%Let $\eta_G$ be a conjugation on $G$ with real form $G_\bbR$
%such that 
%the pair $(\eta_G,\eta_{\textbf M})$ is compatible with the 
%$G$-action on $\textbf M$.
%Then the conjugation~\eqref{conjugation} is $G_\bbR$-equivaraint. 

\subsection{Compatibility with  symmetries}
Recall the $\bbR$-subspace 
$Z\subset Z_\bC$.
For any $s\in[0,2\pi]$, let 
\[q_s=\cos(s)i+\sin(s)k\in\on{Sp}(1).\]
A direct computation shows that 
$\on{Ad}_{q_s}$ preserves the subspace 
$Z=\bbR j\otimes_\bbR Z\subset\on{Im}\mathbb H\otimes_\bbR Z$\footnote{Note that
$Z_\bC=(\bbR j+\bbR k)\otimes_\bbR Z$ is not stable under the family of maps 
$\on{Ad}_{q_s}$.} and its restriction to $Z$ is given by $-\id_Z$.
Consider the family of hyper-K$\dota$hler quotients 
\beq
\frak M_{Z}=\mu_\bbR^{-1}(0)\cap\mu_\bbC^{-1}(-Z)/H_u
\eeq
over $Z$.
Then the discussion in the previous section shows that 
there is a family of  
maps 
\beq
\phi_s=\phi(q_s):\frak M_{Z}\to \frak M_{Z}\ \ \ \ s\in[0,2\pi]
\eeq
making the following diagram commute
\beq\xymatrix{\frak M_Z\ar[r]^{\phi_s}\ar[d]^{}&\frak M_Z\ar[d]^{}\\
Z\ar[r]^{-\id_Z}&Z}.
\eeq
Consider $j\in\on{Sp}(1)$.
Since $\on{Ad}_j=\id_Z$ on $Z$ we have a map
\beq
\phi(j):\frak M_Z\to\frak M_Z
\eeq
making the following diagram commute
\beq\xymatrix{\frak M_Z\ar[r]^{\phi(j)}\ar[d]^{}&\frak M_Z\ar[d]^{}\\
Z\ar[r]^{\id_Z}&Z}.
\eeq
Note that 
\beq
\phi_s^2=\phi(j)^2=\phi(-1).
\eeq
In particular, if $\phi(-1)$ is equal to the identity map 
then $\phi_s$ and $\phi(j)$ are involutions on 
$\frak M_Z$.

Our next goal is to study the compatibility between
the maps $\phi_s$, $\phi(j)$, and the conjugation 
$\eta_{Z_\bC}$
introduced in Section \ref{family of conjugations}.

\begin{definition}\label{quaternionic}
Let $\eta_H$ and $\eta_{\textbf M}$ be conjugations on 
$H$ and $\textbf M$ respectively.
We say that  $\eta_H$ and $\eta_{\textbf M}$
are compatible with the unitary quaternionic representation 
of $H_u$ on $\textbf M$ if the following holds:
\begin{enumerate}
\item the pair $(\eta_H$,$\eta_{\textbf M})$ is compatible with 
the symplectic representation of $H$ on $\textbf M$ (see Definition \ref{symplectic})
\item $\eta_{\textbf M}$ preserves the inner product $(,)$, that is, we have 
$(\eta_{\textbf M}(v),\eta_{\textbf M}(v'))=(v,v')$ for $v,v'\in\textbf M$.
\item $\eta_H$ commutes with the Cartan conjugation $\delta_H$.
\end{enumerate}
\end{definition}

\begin{prop}\label{compatibility of maps}
Let $\eta_H$ and $\eta_{\textbf M}$ be 
conjugations on $H$ and $\textbf M$ compatible with 
the unitary quaternionic representation of $H_u$ on $\textbf M$.
Let $\eta_{Z_\bC}:\frak M_{Z_\bC}\to\frak M_{Z_\bC}$
be the conjugation in~\eqref{conjugation}\footnote{$\eta_{Z_\bC}$ is well-defined since 
$\eta_H$ and $\eta_{\textbf M}$ are compatible with the symplectic representation 
of $H$ on $\textbf M$.}.
Then the subspace $\frak M_Z\subset\frak M_{Z_\bC}$ is stable under 
$\eta_{Z_\bC}$. Moreover if we denote by 
\beq\label{real conjugation}
\eta_Z:\frak M_Z\to \frak M_Z
\eeq the resulting map, we have the following equality of maps on $\frak M_Z$
\beq\label{equality}
\phi_s\circ\eta_Z=\phi(-1)\circ\eta_Z\circ\phi_s\ \ \ \ 
\phi_s\circ\phi(j)=\phi(-j)\circ\phi_s\ \ \ \ \phi(j)\circ\eta_Z=\eta_Z\circ \phi(j)
\eeq

\end{prop}
\begin{proof}
Since $\eta_H$ commutes with the Cartan conjugation $\delta_H$,
the center of $\frak h_u$, and hence its real dual $Z\subset\frak h_u^*$, 
is stable under $\eta_H$ and~\eqref{proj to Z_C} implies that 
$\frak M_Z$ is preserved by the conjugation $\eta_{Z_\bC}$.

We claim that conditions (1) and (2) in Definition \ref{quaternionic}
imply that $\eta_{\textbf M}$ commutes with $J$
and preserves 
$\mu_\bbR^{-1}(0)$.
Assume the claim for the moment.
Then using the equality $I\circ\eta_{\textbf M}=-\eta_{\textbf M}\circ I$
and $K=IJ$, a direct computation shows that we have the following equality of maps on 
$\mu_\bbR^{-1}(0)\cap\mu_\bC^{-1}(-Z)$:
\[(\cos(s)I+\sin(s)K)\circ\eta_{\textbf M}=-\eta_{\textbf M}\circ (\cos(s)I+\sin(s)K)\]
\[(\cos(s)I+\sin(s)K)\circ J=-J\circ (\cos(s)I+\sin(s)K)\]
\[J\circ\eta_{\textbf M}=\eta_{\textbf M}\circ J\]
compatible with the $H_u$-action.
The desired equality~\eqref{equality} follows.

Proof of the claim. 
For any $\xi\in\frak h_u$ and $v\in\textbf M$, we have 
\[\langle\xi,\mu_\bbR(\eta_{\textbf M}(v))\rangle=(I\xi\eta_{\textbf M}(v),\eta_{\textbf M}(v))=
-(\eta_{\textbf M}(I\eta_H(\xi)v),\eta_{\textbf M}(v))=-(I\eta_H(\xi)v,v)=
\langle-\eta_H(\xi),\mu_\bbR(v)\rangle=\]
\[=\langle\xi,-\eta_H(\mu_\bbR(v))\rangle\]
Thus we have $\mu_\bbR(\eta_{\textbf M}(v))=-\eta_H(\mu_\bbR(v))$
and it follows that  $\mu_\bbR^{-1}(0)$ is stable under the 
conjugation $\eta_{\textbf M}$.
Recall that  
$\omega_\bC(v,v')=(Jv,v')+i(Kv,v')$. Thus the equality
$\omega_\bC(\eta_{\textbf M}(v),\eta_{\textbf M}(v'))=\overline{\omega_\bC(v,v')}$
is equivalent to 
\[(J\eta_{\textbf M}(v),\eta_{\textbf M}(v'))+i(K\eta_{\textbf M}(v),\eta_{\textbf M}(v'))=(Jv,v')-i(Kv,v')\]
which 
implies 
\[(J\eta_{\textbf M}(v),\eta_{\textbf M}(v'))=(Jv,v').\]
Since $\eta_{\textbf M}$ preserves $(,)$,
the above equality implies 
\[(J\eta_{\textbf M}(v),\eta_{\textbf M}(v'))=(\eta_{\textbf M}J(v),\eta_{\textbf M}(v'))\]
and it follows that 
$J\circ\eta_{\textbf M}=\eta_{\textbf M}\circ J$.
This finishes the proof of the claim.
 \end{proof}

\begin{remark}\label{commutes with J}
The proof above shows that condition (2)  in Definition \ref{quaternionic} is equivalent to 
the condition that  $ \eta_{\bfM}$ commutes with $J$.

\end{remark}

\subsection{A family of involutions}\label{involutions}
Let $\eta_H$ and $\eta_{\textbf M}$ be 
conjugations on $H$ and $\textbf M$ compatible with 
the unitary quaternionic representation of $H_u$ on $\textbf M$.
Let $G$ be another complex reductive group with a compact real from
$G_u$ and let $\eta_G$ be a conjugation on $G$ with real form 
$G_\bbR$.
Suppose that $\bfM$ is a unitary quaternionic representation of the larger 
group $H_u\times G_u$ and
the conjugations $\eta_H\times\eta_G$ and 
$\eta_{\bfM}$ are compatible with the unitary quaternionic representation.
Then the maps 
$\eta_Z$, $\phi_s$, $\phi(j)$
in Proposition \ref{compatibility of maps} are 
$K_\bbR$-equivaraint
where $K_\bbR=G_\bbR\cap G_u$ is a maximal compact 
subgroup of $G_\bbR$.

\begin{definition}\label{def of involutions}
Consider the following maps
\begin{enumerate}
\item
$\alpha_a=\phi_{\frac{a\pi}{2}}\circ\eta_Z:\frak M_Z\to\frak M_Z,\ \ \ a\in[0,1]$
\item
$\beta=\phi(j)\circ\eta_Z:\frak M_Z\to\frak M_Z$.
\end{enumerate}
\end{definition}

\begin{proposition}\label{alpha beta commute}
We have 
$\alpha_a\circ\beta=\beta\circ\alpha_a
$
for all $a\in[0,1]$.
\end{proposition}
\begin{proof}
Set $s=\frac{a\pi}{2}$.
By Proposition \ref{compatibility of maps},
we have \[\alpha_a\circ\beta=\phi_s\circ\eta_Z\circ\phi(j)\circ\eta_Z=\phi_s\circ\phi(j)\]
and \[\alpha_a\circ\beta=\phi(j)\circ\eta_Z\circ\phi_s\circ\eta_Z=
\phi(j)\circ\phi(-1)\circ\phi_s=\phi_s\circ\phi(j).\]
The result follows.
 \end{proof}

\begin{proposition}\label{family of involutions quiver}
The continuous family of maps 
\[\alpha_a:\frak M_Z\lra\frak M_Z,\ \ \ \ a\in[0,1]\]
satisfies the following:
\begin{enumerate}
\item $\alpha_a^2$ is equal to identity, for all $a\in[0,1]$.
\item $\alpha_a$ is 
$K_\bbR$-equivariant and commutes with the $\bbR^\times$-action.
 \item $\alpha_a$ 
commutes with the projection map $\frak M_Z\to Z$ and induces involutions on the fibers 
$\alpha_a:\frak M_{\zeta_\bC}\to\frak M_{\zeta_\bC}, \zeta_\bC\in Z$ preserving the stratification 
$\frak M_{\zeta_\bC}=\bigsqcup_{(L)}\frak M_{\zeta_\bC,(L)}$.
\item At $a=0$, we have $\alpha_0=\phi(i)\circ\eta_Z$ which is an
anti-holomorphic involution.
\item At $a=1$, we have $\alpha_1=\phi(k)\circ\eta_Z$ which is a
holomorphic involution.

\end{enumerate}
\end{proposition}
\begin{proof}
Note that $q_s^2=-1$ and hence 
$\phi_s^2=\phi(q_s^2)=\phi(-1)$. By
Proposition \ref{compatibility of maps}
we have \[(\phi_s\circ\eta_Z)^2=\phi_s\circ\eta_Z\circ\phi_s\circ\eta_Z=\phi_s^2\circ
\phi(-1)\circ\eta_Z^2=\id.\]
Part (1) follows.
Part (2),(3),(4),(5) follow from the construction.
\end{proof}

\begin{proposition}\label{Cartan involution}
The map 
\[\beta:\frak M_Z\lra\frak M_Z\]
satisfies the following:
\begin{enumerate}
\item $\beta^2=\phi(-1)$.
\item $\beta$ is 
$K_\bbR$-equivariant and commutes with the $\bbR^\times$-action.
\item $\beta$ induces a holomorphic map between fibers
$\beta:\frak M_{\zeta_\bC}\to\frak M_{-\zeta_\bC}$ which takes the stratum  
$\frak M_{\zeta_\bC,(L)}$ to the stratum $\frak M_{-\zeta_\bC,(L)}$.
\end{enumerate}
\end{proposition}
\begin{proof}
Since 
$\beta^2=\phi(j)\circ\eta_Z\circ\phi(j)\circ\eta_Z=\phi(j)^2=\phi(-1)$, part (1) follows.
Part (2), (3) follow from the construction.
\end{proof}

\begin{remark}
Unlike the family of involutions $\alpha_s$, the map 
$\beta$ is well-defined on the whole family $\frak M_{Z_\bC}$.
\end{remark}

%%%%%%
\quash{

\subsection{A variation}
Let 
$g\in\on{O}(V_\bbR)$ be an element such that 
$g^2=\pm1$.
Then the action map
$g:\frak M_Z\to\frak M_Z$ 
commutes with the involutions $\alpha_a$ and $\omega$.

\begin{definition}\label{even}
Let $S\subset\bC\otimes Z$ be a $\bbR$-linear subspace. 
We say that the family of quiver varieties 
$\frak M_S$ is \emph{even} if 
$\phi(-1)=\id:\frak M_S
\to\frak M_S$. 
\end{definition}
\begin{example}
(1) 
If $W=0$ then $\frak M_S$ is even.
(2) If $Q=(Q_0,Q_1)$ is an $\mathrm A_n$-quiver with 
$W_k=0$ for $k\geq 2$ then $\frak M_S$ is even.

\end{example}
}
%%%%%%

\subsection{A stratified homeomorphism}\label{stratified}
Our aim is to trivialize the fixed-point of the family involution $\alpha_a$.
To that end, we will invoke the following lemma.

Recall that a subset $S$ of a 
real analytic manifold $M$ (resp. real algebraic variety $M$)
is called semi-analytic (resp. semi-algebraic) if
any point $s\in S$ has a open neighbourhood $U$ (resp. a Zariski affine open neighbourhood $U$)
such that the intersection $S\cap U$ is a finite union of sets of the form
\[\{x\in U|f_1(x)=\cdot\cdot\cdot=f_r(x)=0, g_1(x)>0,..., g_l(x)>0\},\]
where the $f_i$ and $g_j$ are real analytic functions on 
$U$ (resp. polynomial functions on $U$).

\begin{lemma}\label{fixed points}
Let $M$ and $N$ 
be two 
semi-analytic sets
and 
let $f:M\to N$ be a continuous map.
Let \[\alpha_a:M\to M,\ \ a\in[0,1]\] be a continuous family of involutions over $N$.

\begin{enumerate}
\item
Assume $\alpha_a$ preserves a semi-analytic
stratification of
$M$\footnote{A stratification of a semi-analytic set 
is called semi-analytic if each stratum is a real analytic manifold.} and restricts to a real analytic map on each stratum.
Then the fixed-points of the strata are real analytic manifolds and 
the $\alpha_a$-fixed points $M^{\alpha_a}$ is stratified by the 
fixed-points of the strata.
\item
Assume further that there is a continuous $\bbR_{>0}$-action on
$M$ (resp. $N$) real analytic on strata
and a 
proper continuous 
map
$\lvert\lvert-\lvert\lvert:M\to\bbR_{\geq0}$ such that 
(i) $f:M\to N$ is $\bbR_{>0}$-equivariant
(ii)  
the $\bbR_{>0}$-action on $M$ has 
 a unique fixed point  $o_M\in M$, which is also a stratum
(iii) $\lvert\lvert tm\lvert\lvert=t\lvert\lvert m\lvert\lvert$ 
and $\lvert\lvert \alpha_a(m)\lvert\lvert=\lvert\lvert m\lvert\lvert$ 
for $t\in\bbR_{>0}, a\in[0,1],m\in M$.
Then for any $a,a'\in[0,1]$ there is a $\bbR_{>0}$-equivariant  stratified
homeomorphism
\beq\label{homeo for fibers}
M^{\alpha_s}\is M^{\alpha_{s'}}
\eeq
which 
is real analytic on each stratum
and compatible with the natural maps to $N$.
\item
Assume further that there is a continuous action of a
compact group $L$ on $M$ satisfying 
(i) the action commutes with the 
map $f:M\to N$, the 
involutions $\alpha_a$, and the
$\bbR_{>0}$-action, 
and 
 is 
real analytic on each stratum
(ii) the map 
$\lvert\lvert -\lvert\lvert:M\to\bbR_{\geq 0}$ is $L$-invariant.
Then the homeomorphism in~\eqref{homeo for fibers}
is $L$-equivariant.

\end{enumerate}
\end{lemma}
\begin{proof}
Proof of (1). Only the first claim requires a proof 
and it follows from the general fact that 
the fixed points $M^\alpha$ of a real analytic involution $\alpha$
on a real analytic manifold $M$
is again a real analytic manifold. 

Proof of (2).
\text{Step 1}. Let $M_0=M\setminus\{o_M\}$
and 
$C=\{m\in M_0|\lvert\lvert m\lvert\lvert=1\}$.
Since $\lvert\lvert-\lvert\lvert:M\to\bbR_{\geq0}$ is $\alpha_a$-invariant and proper,
$C$ is compact and stable under the $\alpha_a$-action.
Since $\bbR_{>0}$ acts freely on $M_0$ and $\lvert\lvert-\lvert\lvert$
is $\bbR_{>0}$-equivariant, the 
restriction $\lvert\lvert-\lvert\lvert|_{M_0}:M_0\to\bbR_{>0}$
is a stratified submersion (where $\bbR_{>0}$ is equipped with the trivial stratification).
It follows that $C=\lvert\lvert-\lvert\lvert^{-1}(1)\subset M_0$ is stratified by the intersection of the strata with 
$C$. 

\text{Step 2}.
We shall show that there exists a 
stratified
homeomorphism
\beq\label{step 1}
\nu:C^{\alpha_a}\is C^{\alpha_{a'}}
\eeq
which 
is real analytic on each stratum
and is compatible with natural maps to $N$.
Consider the involution
$\alpha:[0,1]\times C\to[0,1]\times C, \alpha(a,m)=(a,\alpha_a(m))$.
Let $w$ be the 
average of the vector field $\partial_a\times 0$ on $[0,1]\times C$
with respect to the $\bZ/2\bZ$-action given by the involution $\alpha$.
Since $[0,1]\times C$ is compact and the $\bZ/2\bZ$-action is real analytic on each stratum, the vector field $w$ is complete and the integral curves of 
$w$ defines the desired stratified homeomorphism 
$\nu:C^{\alpha_a}\is C^{\alpha_{a'}}$, $a,a'\in[0,1]$
between the fibers 
of the $\alpha$-fixed point $([0,1]\times C)^{\alpha}$ along the projection map to $[0,1]$.

\text{Step 3}.
We have a natural map
$M^{\alpha_a}_0\to C^{\alpha_a}$
sending $m$ to $\frac{m}{\lvert\lvert m\lvert\lvert}$.
Consider the following map
\beq\label{step 2}
M_0^{\alpha_a}\to M_0^{\alpha_{a'}}\ \ m\to \lvert\lvert m\lvert\lvert\nu(\frac{m}{\lvert\lvert m\lvert\lvert}).
\eeq
 Note that 
$M^{\alpha_a}$ is homeomorphic to the 
cone $C(M_0^{\alpha_a})=M_0^{\alpha_a}\cup\{o_M\}$ of $M_0^{\alpha_a}$.
Thus by the functoriality of cone, the map~\eqref{step 2} extends to a 
homeomorphism
\beq\label{step 3}
M^{\alpha_a}\to M^{\alpha_{a'}}
\eeq
sending $o_M$ to $o_M$.
It is straightforward to check that~\eqref{step 3} is a $\bbR_{>0}$-equivariant 
stratified homeomorphism 
which 
are real analytic on each stratum and compatible with the natural maps to $N$.
This finishes the proof of part (2).
Part (3) is clear from the construction of~\eqref{step 3}.
\end{proof}

\begin{example}\label{norm function}
We preserve the set-up in Section \ref{involutions}.
The map
$\lvert\lvert-\lvert\lvert:\frak M_Z=\mu_\bbR^{-1}(0)\cap\mu_\bC^{-1}(Z)/H_u\to\bbR_{\geq0}$
given by $\lvert\lvert m\lvert\lvert=(\tilde m,\tilde m)^{\frac{1}{2}}$, 
where $\tilde m\in\mu_\bbR^{-1}(0)\cap\mu_\bC^{-1}(Z)$ is a lift of $m$, is a
$K_\bbR\times\alpha_a$-invariant
proper real analytic map
satisfying 
$\lvert\lvert\phi(t)m\lvert\lvert=t\lvert\lvert m\lvert\lvert, t\in\bbR_{>0}$.
Let 
$\frak M_0=\mu^{-1}(0)/H_u$, $\alpha_a:\frak M_0\to \frak M_0$ be the family of involutions in Proposition \ref{family of involutions quiver}, and 
 $\phi(t):\frak M_0\to\frak M_0$ be the $\bbR_{>0}$-action in~\eqref{G_m action}. 
Denote by
\[\frak M_0(\bbR)=\frakM_0^{\alpha_0}\ \ \ \ 
\frak M_0^{sym}(\bC)=\frak M_0^{\alpha_1}\]
the fixed points of 
$\alpha_0$ and $\alpha_1$ on $\frak M_0$
respectively.
Applying Lemma \ref{fixed points}
to the case
$M=\frak M_0$ with the stratification 
$\frak M_0=\bigsqcup_{(L)}\frak M_{0,(L)}$, $N=0$, $L=K_\bbR$, 
and the restriction 
$\lvert\lvert-\lvert\lvert|_M:M=\frak M_0\to\bbR_{\geq 0}$ of the function 
$\lvert\lvert-\lvert\lvert$ above to $\frak M_0\subset\frak M_Z$,
we see that there is a $K_\bbR\times\bbR_{>0}$-equivariant stratified homeomorphism
\beq
\xymatrix{\frakM_0(\bbR)\ar[r]^{\sim\ }&\frak M_0^{sym}(\bC)}
\eeq
which are real analytic on each stratum.
Note that whereas $\frak M_0^{sym}(\bC)$  is complex analytic
$\frakM_0(\bbR)$ is not, it is 
a real form of $\frak M_0$.

\end{example}

\section{Quiver varieties}\label{quiver varieties}
In this section we consider the examples when
the hyper-K$\dota$hler quotients are Nakajima's quiver varieties.
We show that any quiver variety has 
a canonical 
conjugation called the split conjugation and hence
has a canonical family of 
involutions $\alpha_a$ introduced in
Section \ref{involutions}.
The main reference for quiver varieties is \cite{Nak1}.

\subsection{Split conjugations}\label{split conjugations}
Let $Q=(Q_0,Q_1)$ be a quiver, where 
$Q_0$ is the set of vertices and $Q_1$ is the set of arrows.
For any $Q_0$-graded hermitian 
vector space
$V=\bigoplus_{k\in Q_0} V_k$, we write 
$\GL(V)=\prod_{k\in Q_0}\GL(V_i)$
and $\on{U}(V)=\prod_{k\in Q_0}\on{U}(V_k)$ where 
$\on{U}(V_k)$ is the unitary group associated to the hermitian vector space 
$V_k$.
We denote by 
$\frakgl(V)$ and $\fu(V)$ be the Lie algebras of $\GL(V)$
and $\on{U}(V)$ respectively.

Let $V=\bigoplus_{k\in Q_0} V_k$ and 
$W=\bigoplus_{k\in Q_0} W_k$ be two $Q_0$-graded hermitian vector spaces.
Define
\beq
\textbf M=\textbf M(V,W)=\bigoplus_{h\in Q_1}\Hom(V_{o(h)},V_{i(h)})\oplus
\Hom(V_{i(h)},V_{o(h)})
\bigoplus_{k\in Q_0}\Hom(W_k,V_k)\oplus\Hom(V_k,W_k).
\eeq
Here $o(h)$ and $i(h)$ are the outgoing and incoming vertices of the oriented arrow 
$h\in Q_1$ respectively.

We consider the
$\mathbb H$-vector space structure on 
$\mathrm M$ given by the original complex structure 
$I$ together with the new complex structure 
$J$ given by
\beq\label{J}
J(X,Y,x,y)=(-Y^\dagger,X^\dagger,-y^\dagger,x^\dagger)
\eeq
where $(X,Y,x,y)\in\Hom(V_{o(h)},V_{i(h)})\oplus
\Hom(V_{i(h)},V_{o(h)})
\oplus\Hom(W_k,V_k)\oplus\Hom(V_k,W_k)$
and 
$(-)^\dagger$ is the hermitian adjoint.

The hermitian inner products on $V_k$ and $W_k$
induces a hermitian inner product on
$\Hom(V_k,W_k)$ (resp. $\Hom(V_k,V_{k'})$) given by
$(f,g)=\tr(fg^\dagger)$.
We consider the hermitian inner product on
$\bfM$ induced from the ones on $V_k$ and $W_k$.

Let $H=\GL(V)$ and $G=\GL(W)$
with compact real from $H_u=\on{U}(V)$ and $G_u=\on{U}(W)$.
Then action of
$H\times G=\GL(V)\times\GL(W)$ on $\mathrm M$  give by the formula
\[(g,g')(X,Y,x,y)=(gXg^{-1},gYg^{-1},gx(g')^{-1},g'yg^{-1})\]
defines a unitary quaternionic representation of 
$\on{U}(V)\times\on{U}(W)$ on $\textbf M$.  
The holomorphic symplectic form 
$\omega_\bC$ is given by
\beq\label{symplectic form}
\omega_\bC((X,Y,x,y),(X',Y',x',y'))=\tr(XY'-YX')+\tr(xy'-x'y)
\eeq

We denote by
\beq
\mu:\textbf M\to\on{Im}(\mathbb H)\otimes\fraku(V)^*=\on{Im}(\mathbb H)\otimes\fraku(V)
\eeq
the hyper-Kahler moment map with respect to the $\on{U}(V)$-action.
Here we identify $\fraku(V)$ with its dual space $\fraku(V)^*$ via the above
hermitian inner product.
We have  the following formulas for the real and complex moment maps
\[\mu_\bbR(X,Y,x,y)=\frac{i}{2}(XX^\dagger-Y^\dagger Y+xx^\dagger-y^\dagger y)\in\frak u(V),\]
\[\mu_\bC(X,Y,x,y)=[X,Y]+xy\in\frakgl(V)=\bC\otimes_\bbR\frak u(V).\]
The hyper-K$\dota$hler quotient $\frak M_\zeta$ is 
called the \emph{quiver variety}.

\begin{lemma}
Let $\eta_{V}$ and $\eta_{W}$ be conjugations on 
$V$ and $W$ compatible with the $Q_0$-grading\footnote{That is, we have 
$\eta_V(V_k)=V_k, \eta_W(W_k)=W_k$ for all $k\in Q_0$.}
and let
$\eta_H$, $\eta_G$, and $\eta_{\bfM}$ be the induced conjugations 
on $H=\GL(V)$, $G=\GL(W)$, and $\bfM$ respectively.
Assume  $\eta_H$ and $\eta_G$
commute with the 
Cartan conjugations on $H$ and $G$ given by the hermitian adjoint.
Then the conjugations $\eta_H\times\eta_G$
and $\eta_{\bfM}$ are 
compatible with the unitary quaternionic representation of 
$H_u\times G_u$ on $\textbf M$.
\end{lemma}
\begin{proof}

$\eta_H\times\eta_G$ commutes with the 
Cartan involution on $H\times G$ by assumption.
Using ~\eqref{J} and~\eqref{symplectic form}, it 
is straightforward to check that 
$\eta_{\bfM}$ commutes with $J$ and 
$\eta_H\times\eta_G$ and $\eta_{\bfM}$ are compatible with the
symplectic representation of $H\times G$ on $\bfM$. 
In view of Remark \ref{commutes with J}, we see that 
$\eta_H\times\eta_G$ and $\eta_{\bfM}$ satisfy (1), (2), (3) in Definition \ref{quaternionic}.
The lemma follows.
\end{proof}

Choose $\textbf v=(\textbf v_k)_{k\in Q_0},\textbf w=(\textbf w_k)_{k\in Q_0}\in\bZ^{Q_0}_{\geq 0}$ and let
$\textbf M(\textbf v,\textbf w)=\textbf M(V,W)$ where 
$V=\bigoplus_{k\in Q_0}\bC^{\textbf v_k}$ and $W=\bigoplus_{k\in Q_0}\bC^{\textbf w_k}$ equipped with
 the standard hermitian inner products.
The standard complex conjugations on $V$ and $W$ induce the 
split conjugations on 
$H=\GL(V)$ and $G=\GL(W)$ commuting with the Cartan conjugations, and hence give rise to 
involutions $\eta_H$, $\eta_G$ and $\eta_{\bfM}$ compatible with the unitary 
quaternionic representation. 
We will call the conjugation 
\[\eta_{Z_\bC}:\frak M_{Z_\bC}\to\frak M_{Z_\bC}\]
on the family of quiver varieties $\frak M_{Z_\bC}$ associated to $\eta_H\times\eta_G$ and $\eta_{\bfM}$
the \emph{split conjugation}.

\subsection{Real-symmetric homoemorphisms for quiver varieties}

Let $\on{O}(W_\bbR)=\on{U}(W)\cap\GL(W_\bbR)$
be the real orthogonal group.
By Propositions \ref{compatibility of maps} and \ref{family of involutions quiver}, the split conjugation $\eta_{Z_\bC}$ on 
$\frak M_{Z_\bC}$ preserves the subspace $\frak M_Z\subset\frak M_{Z_\bC}$ and 
gives rise to a
family of $\on{O}(W_\bbR)$-equivariant involutions
\beq\label{involutions on quiver}
\alpha_a:\frak M_Z\to\frak M_Z\ \ \  a\in[0,1]
\eeq
interpolating the anti-holomorphic involution
$\alpha_0=\phi(i)\circ\eta_{Z}$ and the holomorphic involution
$\alpha_1=\phi(k)\circ\eta_{Z}$, 
and preserving the 
strata $\frak M_{\zeta_\bC,(L)}$
of the fiber $\frak M_{\zeta_\bC}$ for  $\zeta_\bC\in Z$.

The involutions in~\eqref{involutions on quiver} restricts to a family of involutions 
$\alpha_a:\frak M_0\to\frak M_0$.
Write $\frak M_0(\bbR)=\frak M_0^{\alpha_0}$
and $\frak M_0^{sym}(\bC)=\frak M_0^{\alpha_1}$ for the fixed-points of 
$\alpha_0$ and $\alpha_1$.
 The intersections of the stratum $\frak M_{0,(L)}$
 with $\frak M_0(\bbR)$ and $\frak M_0^{sym}(\bC)$
are unions of components
\[\frak M_{0,(L)}\cap\frak M_0(\bbR)=\bigsqcup\mO_{l}(\bbR)\ \ \ \ \ \ \ 
\frak M_{0,(L)}\cap\frak M_0^{sym}(\bC)=\bigsqcup\mO_{l}^{sym}(\bC)\]
In \cite[Theorem 1.9]{BeSc}, Bellamy-Schedler proved that the strata $\frak M_{0,(L)}$ are 
symplectic leaves of $\frak M_{0}$. We will call the components 
$\mO_{l}(\bbR)$ and $\mO_{l}^{sym}(\bC)$ above 
the \emph{real symplectic leaves} and \emph{symmetric symplectic leaves}
respectively.

The following proposition follows from 
Example \ref{norm function}:
\begin{thm}\label{KS for quivers}
There is a $\on{O}(W_\bbR)\times\bbR^\times$-equivaraint stratified homeomorphism
\beq\label{homeo for quiver}
\xymatrix{
\frak M_0(\bbR)\ar[r]^{\sim\ }&\frak M_0^{sym}(\bC)}
\eeq
which restricts to real analytic $\on{O}(W_\bbR)$-equivariant isomorphisms
between 
individual real and symmetric symplectic leaves.
The homeomorphism induces a bijection 
\beq\label{bijection for quiver}
\{\mO_{l}(\bbR)\}_{l}\longleftrightarrow\{\mO_{l}^{sym}(\bC)\}_{l}
\eeq
between real and symmetric leaves preserving the closure relation.

\end{thm}

In the next section we shall see that 
the nilpotent cone $\calN_n(\bC)$ in $\frak{gl}_n(\bC)$ is an example of quiver variety and 
the homeomorphism~\eqref{homeo for quiver} in this case becomes an 
$\on{O}_n(\bbR)\times\bbR^\times$-equivariant 
homeomorphism 
\[\calN_n(\bbR)\is\calN^{sym}_n(\bC)\]
between the real nilpotent cone in $\frak{gl}_n(\bbR)$
and the symmetric nilpotent cone in the space of symmetric matrices $\fp_n(\bC)$
and the bijection~\eqref{bijection for quiver} is the well-known Kostant-Sekiguchi bijection between
$\GL_n(\bbR)$-orbits in $\calN_n(\bbR)$ and $\on{O}_n(\bC)$-orbits in 
$\calN^{sym}_n(\bC)$.
Thus one can view~\eqref{homeo for quiver}
as Kostant-Sekiguchi homeomorphisms for quiver varieties.

%%%%%%%%%%
\quash{
\begin{remark}
It is known that 
nilpotent cone $\calN_n(\bbC)\subset\frakgl_n(\bbC)$ is an example of 
quiver variety \[\frak M_0\is\calN_n(\bC)\]
In Theorem \ref{} we shall show that, in this case 
the involutions $\alpha_0$, $\alpha_1$ are given by the negative of 
the conjugation
$\alpha_0(X)=-\bar{X}$ and the transpose $\alpha_1(X)=X^t$ respectively, and the Proposition above implies that 
there is a $\on{O}(W_\bbR)$-equivaraint stratified homeomorphism
\[\calN_n(\bbR)\is\calN_{n}^{sym}(\bC)\]
between the real nilpotent cone and the symmetric nilpotent cone.
\end{remark}
}
%%%%%%%%%%%%%%%%%%

\section{Real-symmetric homeomorphisms for Lie algebras}\label{main results}

\subsection{Main results}
Let us return to the Cartan subgroup $T\subset G$, stable under $\eta$ and $\theta$,  and maximally split with respect to $\eta$.
Let $\ft\subset \fg$ denote its Lie algebra, $\rW_G=N_G(\ft)/Z_G(\ft)$  the Weyl group
and introduce the affine quotient
$\fc = \fg//G =  \Spec(\calO(\fg)^G)\simeq\ft//\rW_G = \Spec( \calO(\ft)^{\rW_G})$.
Let $\chi:\fg\to\fc$ be the natural map.

Next, let $\fa = \ft \cap \fp$ be the $-1$-eigenspace of $\theta$, and write $\fa_\bbR = \fa \cap \fg_\bbR$ for the real form of $\fa$ with respect to $\eta$.
Let $\rW=N_{K_\bbR}(\fa_\bbR)/Z_{K_\bbR}(\fa_\bbR)=N_K(\fa)/Z_K(\fa)$ be the ``little Weyl group", and introduce  the affine quotient
$
\fc_\fp = \fp//K = \Spec( \calO(\fp)^{K})\simeq\fa//{\rW} = \Spec( \calO(\fa)^{\rW})$. Let $\chi_\fp:\fp\to \fc_\fp$ denote the natural map.

%We have a natural map $\fc_\fp=\fp//K\to\fc=\fg//G$. 
Let $\fc_{\fp,\bbR}\subset\fc$ be the image of the natural map
$\fa_{\bbR}\to\fc$.  Since the map $\fa_{\bbR}\to\fc$ is a polynomial map, 
by Tarski-Seidenberg's theorem, its image $\fc_{\fp,\bbR}$ is semi-algebraic.
For example, if $\fg_\bbR=\frak{sl}_2(\bbR)$ then
$\fc=\bC$ and $\fc_{\fp,\bbR}=\bbR_{\leq 0}$.

Consider the following semi-algebraic subsets of $\fg$,
$\fg_\bbR$ and $\fp$:
\beq\label{three families}
\fg'=\fg\times_\fc\fc_{\fp,\bbR}\ \ \ \ \fg_\bbR'=\fg_\bbR\times_\fc\fc_{\fp,\bbR} \ \ \ \ \fp'=\fp\times_\fc\fc_{\fp,\bbR}.
\eeq
We have
\beq\label{real eigenvalues}
\fg_\bbR'=\big\{x\in \fg_\bbR|\text{\ \ eigenvalues of\ }\ad_x\text{\ are real}\big\}
\eeq
\beq
\fp'=\big\{x\in \fp|\text{\ \ eigenvalues of\ }\ad_x\text{\ are real}\big\}
\eeq
Note that $G$, $G_\bbR$ and $K$ naturally act on $\fg'$, $\fg_\bbR'$ and $\fp'$ respectively, and the actions are along the fibers of the natural projections 
\beq\label{projections}
\fg'\to\fc_{\fp,\bbR}\ \ 
\ \fg_\bbR' \to \fc_{\fp,\bbR}\ \ \ 
\fp' \to \fc_{\fp,\bbR}
\eeq

\begin{thm} \label{homeomorphism for g} 
Suppose all simple factors of the complex reductive  Lie algebra $\fg$ are of classical type.
There is a $K_\bbR$-equivariant homeomorphism
\begin{equation}
\xymatrix{
\fg'_\bbR\ar[r]^{\sim}& \fp'
}
\end{equation}
compatible with the natural projections to $\fc_{\fp,\bbR}$. 
Furthermore, the homeomorphism restricts to a real analytic isomorphism between individual $G_\bbR$-orbits and $K$-orbits.
\end{thm}

We deduce the theorem above from the following.

\begin{thm}\label{family of involutions for g}
Suppose all simple factors of the complex reductive  Lie algebra $\fg$ are of classical type.
There is a continuous one-parameter families of maps 
\[\alpha_{a}:\fg'\lra\fg',\ a\in[0,1]\] satisfying the following:
\begin{enumerate}

\item $\alpha_a^2$ is the identity, for all $s\in[0,1]$.
\item At $a=0$, we have $\alpha_0(M)=\eta(M)$.
\item At $a=1$, we have $\alpha_1(M)=-\theta(M)$.
\item $\alpha_a$ is $K_\bbR$-equivaraint and take a $G$-orbit real analytically to a $G$-orbit.
\item $\alpha_a$ commutes the with projection map
$\fg'\to\fc_{\fp,\bbR}$.

\end{enumerate}

\end{thm}

%%%%%%%%%%%%%%
\quash{
Let $\calN_\xi=\chi^{-1}(\xi)$ be the fiber 
of the $\chi:\fg\to\fc$ over $\xi$.
Assume $\xi\in\fc_{\fp,i\bbR}\subset\fc$.
Since $\xi$ is fixed by the involutions on $\fc$
induced by 
$-\eta$ and $-\theta$, the fiber $\calN_\xi$ is stable under $-\eta$ and $-\theta$
and we write 
\[i\calN_{\xi,\bbR}=\calN_{\xi}\cap i\fg_\bbR\ \ \ \ \ \calN_{\xi,\fp}=\calN_\xi\cap\fp.\]
for the fixed points.
There are finitely many 
$G_\bbR$-orbits and $K$-orbits on $i\calN_{\xi,\bbR}$ and $\calN_{\xi,\fp}$
\[i\calN_{\xi,\bbR}=\bigsqcup_l\mO_{i\bbR,l}\ \ \ \ \calN_{\xi,\fp}=\bigsqcup_l\mO_{\fp,l}\]
Theorem \ref{homeomorphism for g} immediately implies the following:

\begin{thm}\label{generalized KS}
There is a $K_\bbR$-equivaraint stratified homeomorphism
\[i\calN_{\xi,\bbR}\is\calN_{\xi,\fp}\]
which restricts to real analytic isomorphisms between individual $G_\bbR$-orbits and $K$-orbits.
The homeomorphism induces an isomorphism between $G_\bbR$-orbits and $K$-orbits posets
\[|i\calN_{\xi,\bbR}/G_\bbR|\longleftrightarrow|\calN_{\xi,\fp}/K|.\]
\end{thm}
}
%%%%%%%%%%%%%%%%
\subsection{Quiver varieties of type $A$ and conjugacy classes of matrices}
Consider the type $\on{A}_n$ quiver:
\[\xymatrix{Q:\overset{1}\bullet\ar[r]&\overset{2}\bullet\ar[r]&\overset{3}\bullet\ar[r]&\cdot\cdot\cdot
\ar[r]&\overset{n-2}\bullet\ar[r]&\overset{n-1}\bullet\ar[r]&\overset{n}\bullet}\]
Let $\textbf v=(n,n-1,...,2,1)\in\bZ_{\geq0}^n$
and $\textbf w=(n,0,...,0,0)\in\bZ_{\geq0}^n$. 
Consider  
the unitary quaternionic representation $\textbf M(\textbf v,\textbf w)$
of $H_u=\prod_{k=1}^n\on{U}(k)$
in Section \ref{split conjugations}.
A vector in $\textbf M(\textbf v,\textbf w)$ can be represented as 
a diagram 
\beq\label{quiver rep}
\xymatrix{\bC^n\ar@<1ex>[d]^x&&&&&&\\
\bC^n\ar@<1ex>[r]^X\ar@<1ex>[u]^y&\bC^{n-1}\ar@<1ex>[l]^Y\ar@<1ex>[r]^X&\bC^{n-2}\ar@<1ex>[l]^Y\ar@<1ex>[r]^X&\cdot\cdot\cdot\ar@<1ex>[l]^Y\ar@<1ex>[r]^X&\bC^3\ar@<1ex>[l]^Y\ar@<1ex>[r]^X&\bC^2\ar@<1ex>[r]^Y\ar@<1ex>[l]^Y&\bC^1\ar@<1ex>[l]^Y
}
\eeq
Let 
$\frakM_{Z_\bC}=\mu_\bbR^{-1}(0)\cap\mu_\bC^{-1}(-Z_\bC)/H_u\to Z_\bC$ 
be the family of quiver varieties associated to $\textbf M(\textbf v,\textbf w)$.
%The action of $\GL_n(\bC)$ on $\textbf M(\textbf v,\textbf w)$
%given by $g(A,B,a,b)=(A,B,ag^{-1},gb)$ commutes with the $H_u$-action and hence descends to
%a $\GL_n(\bC)$-action on $\frakM_{Z_\bC}$.

Denote by
$\fg_n=\frakgl_n(\bC)$, $\ft_n\subset\fg_n$ the subspace of diagonal matrices,
$\fc_n=\fg_n//\GL_n(\bC)$, and $\chi_n:\fg_n\to\fc_n$ the Chevalley map.
We will fix an identification
$\fc_n=\bC^n$ so that the map $\chi_n:\fg_n\to\fc_n=\bC^n$
is given by $\chi_n(M)=(c_1,...,c_n)$, where   
$T^n+c_1T^{n-1}+\cdot\cdot\cdot+c_n$ is the characteristic polynomial of 
$M$.
Consider the following maps
\beq\label{phi_n}
\tilde\phi_{n,\bC}:\frakM_{Z_\bC}\to\fg_n\times\ft_n\ \ \ \ [X,Y,x,y]\to (yx,\zeta_\bC)
\eeq
\beq\label{iota}
\iota_{n,\bC}:Z_\bC\to\ft_n\ \ \ \zeta_\bC\to (c_1,...,c_n)
\eeq
where $\zeta_\bC=(\zeta_1,...,\zeta_n)$ is the image of  
$[X,Y,x,y]\in\frakM_{Z_\bC}$ under the projection map $\chi_{Z_\bC}:\frak M_{Z_\bC}\to Z_\bC$ 
and $c_i=\zeta_1+\cdot\cdot\cdot+\zeta_i$, $1\leq i\leq n$.
Note that the map $\tilde\phi_{n,\bC}$ intertwines the 
$\GL_n(\bC)\times\bbR^\times$-action on $\frakM_{Z_\bC}$ with the one on
$\fg_n\times\ft_n$ given by $(g,a)(M,t)=(gMg^{-1},a^2t)$.

\begin{prop}\label{MV}
Let $\pi_{Z_\bC}:\frak M_{(\zeta_\bbR,Z_\bC)}\to\frak M_{Z_\bC}$ be the  map in~\eqref{resolution quiver} and
let $\frak M'_{Z_\bC}\subset\frak M_{Z_\bC}$ be its image. 
Assume $\zeta=(\zeta_\bbR,0)$ is generic in the sense of 
\cite[Definition 2.9]{Nak1}. 
\
\begin{enumerate}
\item
The fiber 
$\frakM'_{\zeta_\bC}$ of the projection $\frakM'_{Z_\bC}\to Z_{\bC}$ over $\zeta_\bC\in Z_\bC$
is a 
union of 
strata.
\item 
$\frak M'_{Z_\bC}$ is 
connected and invariant under the $\GL_n(\bC)\times\bbR^\times$-action.
\item
The map $\tilde\phi_{n,\bC}$~\eqref{phi_n} restricts to a $\GL_n(\bC)\times\bbR^\times$-equivariant isomorphism 
\[\phi_{n,\bC}:\frakM'_{Z_\bC}\is\fg_n\times_{\fc_n}\ft_n\]
of complex algebraic varieties making 
the following diagram commute
\[\xymatrix{\frakM'_{Z_\bC}\ar[r]^{\phi_{n,\bC}}\ar[d]&\fg_n\times_{\fc_n}\ft_n\ar[d]\\
Z_\bC\ar[r]^{\iota_{n,\bC}}&\ft_n}.\]
Furthermore,
the map $\phi_{n,\bC}$ induces stratified isomorphisms between individual fibers 
of the projections $\frakM'_{Z_\bC}\to Z_\bC$ and  $\fg_n\times_{\fc_n}\ft_n\to\ft_n$. Here we equipped the fibers of $\fg_n\times_{\fc_n}\ft_n\to\ft_n$ with the 
$\GL_n(\bC)$-orbits stratification.

\end{enumerate}

\end{prop}
\begin{proof}
Part (1) follows from \cite[Corollary 6.11]{Nak1}.
Proof of (2) and (3).
Since 
each stratum $\frakM_{\zeta,(L)}$ is invariant under the 
$\GL_n(\bC)\times\bbR^\times$-action part (1) implies 
$\frakM'_{Z_\bC}$ also has this property.
Moreover, since the $\bbR^\times$-action on $\frak M'_{Z_\bC}$
is a contracting action with a unique fixed point, 
$\frak M'_{Z_\bC}$ is connected.
By the result of Mirkovic-Vybornov \cite[Theorem 6.1]{MV},
which is a generalization the earlier results of Kraft-Procesi \cite{KP} and Nakajima \cite{Nak1},  
the map $\phi_{n,\bC}$ induces isomorphisms between individual fibers 
of the projections $\frakM'_{Z_\bC}\to Z_\bC,\fg_n\times_{\fc_n}\ft_n\to\ft_n$,
and hence is a bijection.
Since $\fg_n\times_{\fc_n}\ft_n$ is normal and $\frak M'_{Z_\bC}$ is connected it follows that 
$\phi_{n,\bC}$ is an isomorphism algebraic varieties.
We claim that $\phi_{n,\bC}$ maps each strata
$\frak M_{\zeta_\bC,(L)}$ isomorphically to a $\GL_n(\bC)$-orbit. For this we observe that 
there are only finitely many $\GL_n(\bC)$-orbits on the fibers of $\fg_n\times_{\fc_n}\ft_n\to\ft_n$
and the closure of any non-closed orbit is singular. 
Since each stratum $\frak M_{\zeta_\bC,(L)}$ is smooth and connected
it follows that 
$\phi_{n,\bC}(\frak M_{\zeta_\bC,(L)})$ is a single $\GL_n(\bC)$-orbit.
The claim follows and the proofs of (2) and (3) are complete.

\end{proof}

\subsection{Reflection functors}
Let $\on{C}=(\on{C}_{kl})_{1\leq k,l\leq n}$ be the Cartan matrix of type $\on{A}_n$.
Identify $Z_\bC$ with $\bC^n$ and consider the reflection representation of 
the Weyl group $\rW$ on $Z_\bC$.
For any simple reflection $s_k, k\in[1,n]$
and $\zeta_\bC=(\zeta_1,...,\zeta_n)\in\bZ_\bC$, we have 
$s_k(\zeta_\bC)=\zeta_\bC'$ where $\zeta'_l=\zeta_l-\on{C}_{kl}\zeta_k$.

In \cite{Nak1}, Nakajima associated to each $k\in[1,n]$
a certain hyper-K\"ahler isometry $S_k:\frak M_{\zeta_\bC}(\textbf v,\textbf w)\is\frak M_{\zeta'_\bC}(\textbf v',\textbf w)$
called the \emph{reflection functor}.
Here $\zeta'_\bC=s_k(\zeta_\bC)$ 
and $\textbf v'$ is given by
$v_k'=v_k-\sum_l\on{C}_{kl}v_l+w_k$, $v_l'=v_l$ if $l\neq k$
for $\textbf v=(v_1,...,v_n)$, $\textbf w=(w_1,...,w_n)$.
Moreover, it is shown in \emph{loc. cit.} that 
the reflection functors $S_k$ satisfy the Coxeter relations of the Weyl group.

In the case $\textbf v=(n,n-1,...,1)$ and $\textbf w=(n,0...,0)$,
a direct calculation shows that, 
for $k\in[2,n]$, 
 we have 
$\textbf v=\textbf v'$ and hence 
$S_k:\frak M_{\zeta_\bC}(\textbf v,\textbf w)\is\frak M_{\zeta'_\bC}(\textbf v,\textbf w)$.
Let $\on{S}_n\subset\rW$ be the subgroup generated by
the simple reflections $s_2,...,s_n$.
As $\zeta_\bC$ varies over $Z_\bC$, the 
reflection functors 
$S_2,..,S_n$ define a 
$\on{S}_n$-action on
$\frak M_{Z_\bC}=\bigcup_{\zeta_\bC\in Z_\bC}\frak M_{\zeta_\bC}(\textbf v,\textbf w)$
such that the projection map 
$\frak M_{Z_\bC}\to Z_\bC$ is
$\on{S}_n$-equivariant.

\begin{lemma}\label{S_n action}
The subset
$\frak M_{Z_\bC}'\subset\frak M_{Z_\bC}$ is 
invariant under the $\on{S}_n$-action and 
the isomorphism 
$\phi_{n,\bC}:\frak M_{Z_\bC}'\is\fg_n\times_{\fc_n}\ft_n$
is $\on{S}_n$-equivariant.
\end{lemma}

\begin{proof}
We first claim that the map $\tilde\phi_{n,\bC}:\frakM_{Z_\bC}\to\fg_n\times\ft_n$~\eqref{phi_n}
is $\on{S}_n$-equivariant.
Recall the isomorphism $\iota_{n,\bC}:Z_\bC\is\ft_n$ in~\eqref{iota}.
A direct computation shows that 
$\iota_{n,\bC}$ intertwines the action of 
$s_k$ and the simple reflection $\sigma_{k-1,k}\in\on{S}_n$
for $k\geq 2$.
On the other hand,
the formula for the reflection functors in \cite[Section 3(i)]{Nak2}
implies that, for any $[X,Y,x,y]\in\frak M_{Z_\bC}$, 
we have $S_k([X,Y,x,y])=[\tilde X,\tilde Y,x,y]$ for $k\geq 2$.
All together we see that 
\[\tilde\phi_{n,\bC}(S_k([X,Y,x,y]))=\tilde\phi_{n,\bC}([X',Y',x,y])=(yx,\sigma_{k-1,k}(\iota_{n,\bC}(\zeta_\bC)))=\sigma_{k-1,k}(yx,\iota_{n,\bC}(\zeta_\bC))=\]
\[=\sigma_{k-1,k}(\tilde\phi_{n,\bC}([X,Y,x,y])).\]
The claim follows.
To complete the proof of the lemma, we need to show that 
$\frak M_{Z_\bC}'$ is $\on{S}_n$-invariant.
Let $Z_{\bC}^0\subset Z_\bC$ (resp. $\ft_n^0\subset\ft_n$ )be the open dense subset consisting of 
vectors with trivial stabilizers in $\on{S}_n$.
The isomorphism $\phi_{n,\bC}$ induces an isomorphism 
$\frak M'_{Z^0_\bZ}\is\fg_n\times_{\fc_n}\ft_n^0$, where
$\frak M'_{Z^0_\bZ}=\frak M_{Z_\bC}'\times_{Z_\bC}Z_{\bC}^0$,
and it follows that 
$\frak M'_{Z_\bC^0}$ is open dense in $\frak M_{Z_\bC}'$
and the 
fibers of the projection $\frak M'_{Z^0_\bZ}\to Z_\bC^0$ 
are smooth. According to
\cite[Theorem 4.1]{Nak2}, the map 
$\pi_{Z_\bC}:\frak M_{\tilde Z_\bC}\to\frak M_{Z_\bC}$
is an isomorphism over $\frak M_{Z^0_\bZ}=\frak M_{Z_\bZ}\times_{Z_\bC}Z^0_{\bC}$ and it
follows that 
$\frak M'_{Z^0_\bZ}=\frak M_{Z_\bC}\times_{Z_\bC}Z_\bC^0$, which is $\on{S}_n$-invariant. On the other hand,
the same argument as in the proof of \cite[Theorem 4.1(1)]{Nak2}
shows that 
the map 
$\pi_{Z_\bC}:\frak M_{(\zeta_\bbR,Z_\bC)}\to\frak M_{Z_\bC}$ is proper and hence its image $\frak M_{Z_\bC}'=\pi_{Z_\bC}(\frak M_{(\zeta_\bbR,Z_\bC)})\subset\frak M_{Z_\bC}$ is a closed subset.
Thus $\frak M_{Z_\bC}'$ is equal to the closure of 
$\frak M_{Z^0_\bC}'$ in $\frak M_{Z_\bC}$ and, as $\frak M_{Z^0_\bC}'$
is $\on{S}_n$-invariant, it implies $\frak M_{Z_\bC}'$ is $\on{S}_n$-invariant.
The lemma follows.

\end{proof}

\subsection{Involutions on the spaces of matrices with real eigenvalues}
Let $\frakM_Z'\subset\frak M_{Z}$ be the image of 
$\pi_Z:\frak M_{\tilde Z}\to\frak M_Z$ and let 
$\fg_n\times_{\fc_n}i\ft_{n,\bbR}$
where $i\ft_{n,\bbR}\subset\ft_n$
is the $\bbR$-subspace consisting of diagonal matrices with pure imaginary entries.
Then the isomorphisms $\phi_{n,\bC}$ and $\iota_{n,\bC}$
above 
restricts to isomorphisms 
\beq\label{res iso}
\frakM_Z'\is\fg_n\times_{\fc_n}i\ft_{n,\bbR}\ \ \ \ \ \ \  Z\is i\ft_{n,\bbR}
\eeq
Consider the family of involutions
$\alpha_a:\frakM_{Z}\to\frakM_Z$ in Proposition \ref{family of involutions quiver}
associated to the split conjugations in Section \ref{split conjugations} 
and the map $\beta:\frak M_Z\to\frak M_Z$ in Proposition \ref{Cartan involution}.
Note that the action of $-1\in\bbR^\times$ on $\frak M_Z$ is trivial
(it becomes the action of $1=(-1)^2$ on
$\fg_n\times_{\fc_n}i\ft_{n,\bbR}$) thus, by Proposition \ref{Cartan involution} (1),
$\beta$ is an involution. 
Note also that the fibers of the projection $\frakM_Z'\to Z$ 
are union of strata (Proposition \ref{MV} (1)), thus Proposition \ref{family of involutions quiver} (3) and Proposition \ref{Cartan involution} (3)
imply that 
$\frak M_Z'$ is invariant under the involutions $\alpha_a$ and $\beta$.

To relate $\frak M_Z'$ with matrices with real eigenvalues let us consider the 
following composition 
\beq
\phi_n:\frak M_Z'\stackrel{~\eqref{res iso}}\is\fg_n\times_{\fc_n}i\ft_{n,\bbR}\is\fg_n\times_{\fc_n}\ft_{n,\bbR}
\eeq
\beq 
\iota_n:Z\stackrel{
~\eqref{res iso}}\is i\ft_{n,\bbR}\is\ft_{n,\bbR}
\eeq
where the second isomorphisms are given by 
$\fg_n\times_{\fc_n}i\ft_{n,\bbR}\is\fg_n\times_{\fc_n}\ft_{n,\bbR}, (x,v)\to (ix,iv)$
and 
$i\ft_{n,\bbR}\to\ft_{n,\bbR}, v\to iv.$
Note that  
 the following diagram is commutative
\beq\label{MV iso over Z}
\xymatrix{\frakM'_{Z}\ar[r]^{\phi_n\ \ \ \ }\ar[d]&\fg_n\times_{\fc_n}\ft_{n,\bbR}\ar[d]\\
Z\ar[r]^{\iota_n}&\ft_{n,\bbR}}
\eeq
where the vertical arrows are the natural projections.

Now the isomorphism $\phi_n:\frak M_Z'\is\fg_n\times_{\fc_n}\ft_{n,\bbR}$  gives rise to involutions on $\fg_n\times_{\fc_n}\ft_{n,\bbR}$: 
\beq\label{induced map 1}
\tilde\alpha_{n,a}=\phi_n\circ\alpha_a\circ\phi_n^{-1}:\fg_n\times_{\fc_n}\ft_{n,\bbR}\to\fg_n\times_{\fc_n}\ft_{n,\bbR} \ \ a\in[0,1]
\eeq
\beq\label{induced map 2}
\tilde\beta_n=\phi_n\circ\beta\circ\phi_n^{-1}:\fg_n\times_{\fc_n}\ft_{n,\bbR}\to\fg_n\times_{\fc_n}\ft_{n,\bbR}
\eeq

\begin{lemma}\label{S_n action}
\
\begin{enumerate}
\item
The involutions $\tilde\beta_n$
is given 
by $\tilde\beta_n(M,v)=(-M^t,-v)$.
In particular, $\tilde\beta_n$
commutes with the action of the 
symmetric group $\on{S}_n$ on $\fg_n\times_{\fc_n}\ft_{n,\bbR}$.
\item
The involution
$\tilde\alpha_{n,a}$ commutes with the action of the 
symmetric group $\on{S}_n$ on $\fg_n\times_{\fc_n}\ft_{n,\bbR}$.

\end{enumerate}
\end{lemma}

\begin{proof}
Let $(M,v)\in\fg_n\times_{\fc_n}\ft_{n,\bbR}$.
Choose $[X,Y,x,y]\in\frak M'_{\zeta_\bC}
$ such that 
\[\phi_n([X,Y,x,y])=i(yx,\iota_{n}(\zeta_\bC))=(M,v).\]
According to Definition \ref{def of involutions} and Proposition \ref{Cartan involution},
we have 
\[\beta([X,Y,x,y])=\phi(j)\circ\eta_Z([X,Y,x,y])=[-\bar Y^\dagger,\bar X^\dagger,
-\bar y^\dagger,\bar x^\dagger]\in\frak M_{-\zeta_\bC}'\]
It follows that 
\[
\tilde\beta_n((M,v))=\phi_n([-\bar Y^\dagger,\bar X^\dagger,
-\bar y^\dagger,\bar x^\dagger])=i((\bar x^\dagger)(-\bar y^\dagger),-\iota_n(\zeta_\bC))=(-M^t,-v).
\]
Part (1) follows.

According to Definition \ref{def of involutions},
we have 
\[\alpha_a([X,Y,x,y])=(\cos(s)\phi(i)+\sin(s)\phi(k))\circ\eta_Z([X,Y,x,y])=[X',Y',x',y']\in\frak M_{\zeta_\bC}'.\]
where 
\[x'=i\cos(s)\bar x-i\sin(s)\bar y^\dagger,\ \ \ \ y'=i\cos(s)\bar y+i\sin(s)\bar x^\dagger
.\] 
On the other hand, we have 
$S_k([X,Y,x,y])=[\tilde X,\tilde Y,x,y]$. 
Thus
\beq\label{x,y}
\alpha_{a}\circ S_k([X,Y,x,y])=[(\tilde X)',(\tilde Y)',x',y'],\ \ \ \ 
S_k\circ\alpha_{a}([X,Y,x,y])=[\tilde{(X')},\tilde{(Y')},x',y'].
\eeq
Since $\phi_n$ commutes with the $\on{S}_n$-action (Lemma \ref{S_n action}), 
we obtain
\[\tilde\alpha_{n,a}\circ S_k((M,v))=\tilde\alpha_{n,a}\circ S_k\circ\phi_n([X,Y,x,y])=
\phi_n\circ\alpha_a\circ S_k([X,Y,x,y])=i(y'x',s_k(v))\]
\[S_k\circ\tilde\alpha_{n,a}((M,v))=
S_k\circ\alpha_a\circ\phi_n([X,Y,x,y])=\phi_n\circ S_k\circ\alpha_a([X,Y,x,y])=i(y'x',s_k(v)).\]
Part (2) follows.
The proof is complete.

\end{proof}

Let $\fc_{n,\bbR}\subset\fc_n$
be the image of the map $\ft_{n,\bbR}\to\fc_n$ 
and let
$\fg_n'=\fg_n\times_{\fc_n}\fc_{n,\bbR}\subset\fg_n$.
Note that both $\fc_{n,\bbR}$ and $\fg_n'$ are semi-algebraic sets. 
We have 
\beq
\fg_n'=\big\{x\in\fg_n|\text{\ \ eigenvalues of\ }x\text{\ are real}\big\}.
\eeq
Since the natural map 
$\fg_n\times_{\fc_n}\ft_{n,\bbR}\to\fg_n'=\fg_n\times_{\fc_n}\fc_{n,\bbR}$
is $\on{S}_n$-equivariant (where $\on{S}_n$-acts trivially on $\fg_n'$), 
Lemma \ref{S_n action} implies that the involutions $\tilde\alpha_{n,a}$ and 
$\tilde\beta_n$
in \eqref{induced map 1}  and~\eqref{induced map 2} 
descend to a continuous family of 
involutions
on $\fg_n'$:
\beq
\alpha_{n,a}:\fg_n'\to\fg_n'
\eeq
compatible with projections to 
$\fc_{n,\bbR}$ and an involution
\beq
\beta_n:\fg_n'\to\fg_n'.
\eeq
Moreover, $\beta_n$ is equal to the restriction of the Cartan involution on $\fg_n$ to $\fg_n'$: 
\beq\label{beta=Cartan}
\beta_n(M)=-M^t
\eeq

\begin{thm}\label{family of involutions for g_n'}
The continuous one-parameter families of maps
\[\alpha_{n,a}:\fg_n'\lra\fg_n',\ a\in[0,1]\] satisfying the following:
\begin{enumerate}
\item $\alpha_{n,a}^2$ is equal to the identity map, for all $a\in[0,1]$.
\item At $a=0$, we have $\alpha_{n,0}(M)=\overline M$.
\item At $a=1$, we have $\alpha_{n,1}(M)=M^t$.
\item $\alpha_{n,a}$ is $\on{O}_n(\bbR)$-equivaraint and take a $\GL_n(\bC)$-orbit real analytically to itself.
\item $\alpha_{n,a}$ commutes both with the Cartan involution $\beta_n$
and with the projection map
$\fg_n'\to\fc_{n,\bbR}$, for all $a\in[0,1]$.

\end{enumerate}

\end{thm}

\begin{proof}
Part (1) follows from the construction and Part (5) follows from the 
commutative diagram~\eqref{MV iso over Z}.
Let $\phi'_n:\frak M_Z'\stackrel{\phi_n}\is\fg_n\times_{\fc_{n}}\ft_{n,\bbR}\to\fg'_n$
where the last map is given by
$\fg_n\times_{\fc_{n}}\ft_{n,\bbR}\to\fg_n\times_{\fc_{n}}\fc_{n,\bbR}=\fg_n'$.
Let $M\in\fg_n'$.
Choose
$[X,Y,x,y]\in\frak M_Z'$ such that 
\[M=\phi_n'([X,Y,x,y])=iyx\]
By Definition \ref{def of involutions}, we have 
\[\alpha_{0}([X,Y,x,y])=\phi(i)\circ\eta_Z([X,Y,x,y])=[i\overline X,i\overline Y,i\bar x,i\bar y]\]
\[\alpha_{1}([X,Y,x,y])=\phi(k)\circ\eta_Z([X,Y,x,y])=[-i\bar Y^\dagger,i\bar X^\dagger,-i\bar y^\dagger,i\bar x^\dagger].\]
It follows that 
\[
\alpha_{n,0}([X,Y,x,y])=\phi'_n([i\overline X,i\overline Y,i\bar x,i\bar y])=i(-\bar y\bar x)=\overline M
\]
\[
\alpha_{n,1}([X,Y,x,y])=\phi_n'([-i\overline Y^\dagger,i\overline X^\dagger,-i\bar y^\dagger,i\bar x^\dagger])=
i(\bar x^\dagger\bar y^\dagger)=i(\overline{(yx)}^\dagger)=
i(yx)^t=M^t.
\]
Part (2) and (3) follow.

By Proposition \ref{MV} (3), the isomorphism 
$\phi_n:\frak M_Z'\to\fg_n\times_{\fc_n}\ft_{n,\bbR}$
maps each stratum $\frak M_{\zeta_\bC,(L)}$ real analytically to a 
$\on{GL}_n(\bC)$-orbit. Now part (4) follows from the 
fact the involution $\alpha_a$ on $\frak M_Z'$ is 
$\on{O}_n(\bbR)$-equivariant and $\frak M_{\zeta_\bC,(L)}$
is invariant under $\alpha_a$.

\end{proof}

Let $\fg'_{n,\bbR}$ be the space of
$n\times n$ real matrices with real eigenvalues.
 Let $\fp_n'$ be the space of $n\times n$ symmetric matrices with 
real eigenvalues.
It is clear that 
$\fg'_{n,\bbR}=(\fg'_n)^{\alpha_0}$ and $\fp_n'=(\fg_n')^{\alpha_1}$.

\begin{thm} \label{homeomorphism for g_n} 
There is an $\on{O}_n(\bbR)\times\bbR^\times$-equivariant homeomorphism
\begin{equation}
\xymatrix{
\fg'_{n,\bbR}\ar[r]^{\sim}&\fp_n'
}
\end{equation}
compatible with the natural projections to $\fc_{n,\bbR}$. 
Furthermore, the homeomorphism restricts to a real analytic isomorphism between individual $\GL_n(\bbR)$-orbits and $\on{O}_n(\bbC)$-orbits.
\end{thm}
\begin{proof}

Consider the Lusztig stratification of $\fg_n$. 
The stratum through $g$ with a Jordan decomposition 
$g=s+u$ consists of all $\GL_n(\bC)$-orbits through 
$u+Z_r(\fl)$ where $\fl=Z_{\fg_n}(s)$ is the centralizer of 
$s$ in $\fg_n$ and $Z_r(\fl)=\{x\in Z(\fl)| Z_{\fg_n}(x)=\fl\}$
is the regular part of the center $Z(\fl)$ of $\fl$.
It is clear that the Lusztig stratification
restricts to the orbits stratifications on the fibers of the Chevalley map
$\chi_n:\fg_n\to\fc_n$ and a stratification on 
$\fg_n'=\fg_n\times_{\fc_n}\fc_{n,\bbR}$. 

Recall the $\on{U}(n)$-invariant function $\lvert\lvert-\lvert\lvert:\frak M_Z\to\bbR_{\geq0}$
in Example \ref{norm function}. 
The restriction of $\lvert\lvert-\lvert\lvert$  along the closed embedding $\fg_n\times_{\fc_n}\ft_n\stackrel{\phi_n}\is\frak M_Z'\subset\frak M_Z$ gives rise to 
a function $\fg_n\times_{\fc_n}\ft_n\to\bbR_{\geq0}$.
Its average respect to the $\on{S}_n$-action on $\fg_n\times_{\fc_n}\ft_n$
defines a $\on{S}_n$-invariant function 
$\fg_n\times_{\fc_n}\ft_n\to\bbR_{\geq0}$ which descends to a function
$\lvert\lvert-\lvert\lvert_{\fg_n'}:\fg'_n\to\bbR_{\geq0}$. 
It follows from Theorem \ref{family of involutions for g_n'} and the construction of $\lvert\lvert-\lvert\lvert_{\fg_n'}$ that the function 
$\lvert\lvert-\lvert\lvert_{\fg_n'}$ together with
 the real analytic map $\fg_n'\to\fc_{n,\bbR}$  and the Lusztig stratification on $\fg'_n$  
 satisfy the assumption in Lemma \ref{fixed points}, and hence we obtain a stratified 
$\on{O}_n(\bbR)$-equivariant homeomorphism 
\beq\label{key iso}
\fg_{n,\bbR}'\to\fp'_n
\eeq
which 
are real analytic on each stratum
and compatible with the maps to $\fc_{n,\bbR}$.
Since each stratum in $\fg'_{n,\bbR}$ (resp. $\fp_n'$) is a finite union of $\GL_n(\bbR)$-orbits (resp. $\on{O}_n(\bC)$-orbits)
and $\on{O}_n(\bbR)$-acts simply transitively on connected components 
of each orbits, it follows that 
the homeomorphism~\eqref{key iso} 
restricts to a real analytic isomorphism between individual $\GL_n(\bbR)$-orbits and $\on{O}_n(\bC)$-orbits.
\end{proof}

\subsection{Proof of Theorem \ref{family of involutions for g}}
We shall deduce Theorem \ref{family of involutions for g}
from Theorem \ref{family of involutions for g_n'}.

Let $\fg$ be a simple Lie algebra of classical type with 
real form $\fg_\bbR$.
Recall the classification of real forms of classical types:
\begin{lemma}\cite[Section 4]{OV}\label{classification}
Here is the complete list of all possible quadruple $(\fg_\bbR,\frakk,\eta,\theta)$ (up to isomorphism):
\\
(a) $\fg=\frak{sl}_n(\bC)$:
\begin{enumerate}
\item 
$\frakg_\bbR=\frak{sl}_n(\bbR), \frakk=\frak{so}_n(\bC),\ \eta(g)=\bar g,\ \theta(g)=-g^t$.
\item
$\frakg_\bbR=\frak{sl}_m(\mathbb H), \frakk=\frak{sp}_m(\bC),\ \eta(g)=\on{Ad}S_m(\bar g),\ \theta(g)=-\on{Ad}S_m(g^t)\ \ (n=2m)$.
\item 
$\frakg_\bbR=\frak{su}_{p,n-p}, \frakk=(\frak{gl}_p(\bC)\oplus\frak{gl}_{n-p}(\bC))\cap\frakg,\ \eta(g)=-\on{Ad}I_{p,n-p}(\bar g^t),\ \ \theta(g)=\on{Ad}I_{p,n-p}(g).$
\end{enumerate}
(b) $\frakg=\frak{so}_n(\bC)$:
\begin{enumerate}
\item 
$\frakg_\bbR=\frak{so}_{p,n-p}, \frakk=\frak{so}_p(\bC)\oplus\frak{so}_{n-p}(\bC),\ \eta(g)=\on{Ad}I_{p,n-p}(\bar g),\ \theta(g)=\on{Ad}I_{p,n-p}(g)$.
\item
$\frakg_\bbR=\frak u^*_m(\mathbb H), \frakk=\frak{gl}_m(\bC),\ \eta(g)=\on{Ad}S_m(\bar g),\ \theta(g)=\on{Ad}S_m(g)\ \ (n=2m)$.
\end{enumerate}
(c) $\frakg=\frak{sp}_{n}(\bC)$, $n=2m$:
\begin{enumerate}
\item 
$\frakg_\bbR=\frak{sp}_{2m}(\bbR), \frakk=\frak{gl}_m(\bC),\ \eta(g)=\bar g,\ \theta(g)=\on{Ad}S_{m}(g)$.
\item
$\frakg_\bbR=\frak{sp}_{p,m-p}, \frakk=\frak{sp}_{2p}(\bC)\oplus\frak{sp}_{2m-2p}(\bC),\ \eta(g)=-\on{Ad}K_{p,m-p}(\bar g^t),\ \theta(g)=\on{Ad}K_{p,m-p}(g)$.
\end{enumerate}
Here $S_m=\left(\begin{array}{cc}0&-Id_m\\Id_m&0\end{array}\right)$, 
$I_{p,n-p}=\left(\begin{array}{cc}Id_{p}&0\\0&-Id_{n-p}\end{array}\right)$, and $K_{p,m-p}=\left(\begin{array}{cc}I_{p,m-p}&0\\0&I_{p,m-p}\end{array}\right)$
\end{lemma}

Consider the following commutative diagram
\beq\label{res of adj}
\xymatrix{\fg\ar[r]^{\iota_\fg}\ar[d]^{\chi}&\fg_n\ar[d]^{\chi_n}\\
\fc\ar[r]^{\iota_\fc}&\fc_n}
\eeq
where $
\iota_\fg:\fg\to\fg_n$ is the natural embedding 
and
$\iota_\fc:\fc=\fg//G\to\fc_n=\fg_n//\GL_n(\bC)$.
We have the following explicit description of $
\chi$ and $\iota_\fc$.
For any $M\in\fg_n$, let 
\[T^n+c_1T^{n-1}+c_2T^{n-1}+\cdot\cdot\cdot+c_n\]
be the characteristic polynomial of $M$. 
In the case $\fg=\frak{sl}_n(\bC)$, we have $c_1=0$ and 
one can identify 
$\fc$ with $\bC^{n-1}$ so that 
\[\chi(M)=(c_2,c_3,...,c_n)\]
\[\iota_\fc(c_1,...,c_n)=(0,c_2,...,c_n)\]
In the case $\fg=\frak{sp_n}(\bC)$ or $\frak{so}_n(\bC)$
we have $c_1=c_3=\cdot\cdot\cdot=0$ and one can 
choose an identification 
of $\fc=\bC^{[n/2]}$
such that $\chi:\fg\to\fc=\bC^{[n/2]}$
is given by
\[\chi(M)=(c_2,c_4,...,c_{n})\ \ \ \text{\ \ if\ \ } \fg=\frak{sp}_n(\bC)\]
\[\chi(M)=(c_2,c_4,...,c_{n-1})\ \ \ \text{\ \ if\ \ } \fg=\frak{so}_n(\bC)\ \ \ \ n=2m+1\]
\[\chi(M)=(c_2,c_4,...,c_{n-2},\tilde c_{n})\ \ \ \text{\ \ if\ \ } \fg=\frak{so}_n(\bC)\ \ \ \ n=2m\]
where $\tilde c_n=\on{Pf}(M)$ is the Pfaffian of $M$ satisfying 
$\on{Pf}(M)^2=\on{det}(M)=c_n$,
and the map $\iota_\fc$ is given by 
\[\iota_\fc(c_2,c_4,...,c_{n-1})=(0,c_2,0,c_4,...0,c_{n-1})\ \ \ \text{\ \ if\ \ } \fg=\frak{sp}_n(\bC)\]
\[\iota_\fc(c_2,c_4,...,c_{n-1})=(0,c_2,0,c_4,...0,c_{n-1})\ \ \ \text{\ \ if\ \ }\fg=\frak{so}_n(\bC)\ \ \ \ n=2m+1, l=m\]
\beq\label{Pf}
\iota_\fc(c_2,c_4,...,c_{n-2},\tilde c_{n})=(0,c_2,0,c_4,...,0,\tilde c_{n}^2) \ \ \ \text{\ \ if\ \ }\fg=\frak{so}_n(\bC)\ \ \ \ n=2m,  l=m.
\eeq
\begin{remark}\label{even orthogonal}
It follows that the map $\iota_\fc:\fc\to\fc_n$ is a closed embedding except the case 
$\fg=\frak{so}_n$, $n=2m$.
\end{remark}
Recall the 
semi-algebraic sets $\fc_{\fp,\bbR}\subset\fc$ and 
$\fg'=\fg\times_{\fc}\fc_{\fp,\bbR}\subset\fg$ introduced~\eqref{three families}.
Since for any $x\in\fg'$ the eigenvalues of
$\on{ad}_x$ are real, the embedding 
$\fg'\to\fg_n$ factors through $\fg'\to\fg_n'\subset\fg_n$ and 
diagram~\eqref{res of adj} restricts to a digram 
\beq\label{restricted diagram}
\xymatrix{\fg'\ar[r]^{\iota_\fg}\ar[d]^{\chi}&\fg'_n\ar[d]^{\chi_n}\\
\fc_{\fp,\bbR}\ar[r]^{\iota_\fc}&\fc_{n,\bbR}}.
\eeq

Note that Proposition \ref{alpha beta commute}
implies $\alpha_{n,a}\circ\beta_n=\beta_n\circ\alpha_{n,a}$. 
On the other hand, since 
$S_m,I_{p,n-p},K_{p,m-p}\in\on{O}_n(\bbR)$,  Proposition \ref{family of involutions for g_n'} (4) implies that
the involutions $\on{Ad}S_m$, $\on{Ad}I_{p,n-p}$, and $\on{Ad}K_{p,m-p}$ on $\fg_n'$
commute with both 
$\alpha_{n,a}$ and $\beta_n$.
Now a 
direct computation, 
using
the formula of $\theta$ in Lemma \ref{classification}, 
shows that 
the
compositions
\beq\label{def of involution}
\alpha_{n,s}\circ\beta_n\circ\theta:\fg_n'\to\fg'_n\ \ \ \ s\in[0,1].
\eeq
are involutions.
We claim that the subspace $\fg'\subset\fg_n'$ is invariant under the involutions~\eqref{def of involution}.
Consider the involution $\sigma$ on $\fg_n$ such that 
$(\fg_n)^\sigma=\fg$, that is,
$\sigma$ is
given by 
$\sigma=\beta_n$ if $\fg=\frak{so}_n(\bC)$ and $\sigma=\on{Ad}(S_m)\circ\beta_n$ if 
$\fg=\frak{sp}_n(\bC)$.
Since the map~\eqref{def of involution} commutes with the involution $\sigma$,
the $\sigma$-fixed points 
$(\fg_n')^\sigma$  is invariant under the map~\eqref{def of involution}.
The claim follows.

The diagram~\eqref{restricted diagram} implies that $\fg'$ is equal to the base-change 
\beq\label{sigma}
\fg'=(\fg_n')^\sigma\times_{\fc_{n,\bbR}}\iota_\fc(\fc_{\fp,\bbR}).
\eeq
of $(\fg_n')^\sigma$ to the subspace $\iota_\fc(\fc_{\fp,\bbR})\subset\fc_{n,\bbR}$
and hence the maps~\eqref{def of involution}
restrict  to a family of involutions
\beq
\alpha_a:\fg'\to\fg'\ \ \ \ a\in[0,1].
\eeq

We shall show that the map $\alpha_a$ above satisfies properties 
(1) to (5) in Theorem \ref{family of involutions for g}.
Properties (1), (2), (3) of $\alpha_{n,a}$ in Theorem \ref{family of involutions for g_n'} immediately implies that 
$\alpha_a$ satisfies properties (1), (2), (3) in Theorem \ref{family of involutions for g}.
Property (4) follows from the fact that 
the intersection of an adjoint orbit of $\fg_n$ with $\fg$ is a finite disjoint union 
of $G$-orbits and each $G$-orbit is a connected component.
We now check property (5). We need to show that 
$\alpha_a$ preserves the fibers of $\chi:\fg'\to\fc_{\fp,\bbR}$.
Assume $\fg$ is not of type $\on{D}$. Then by Remark \ref{even orthogonal}, the map 
$\fc_{\fp,\bbR}\to\fc_{n,\bbR}$ is a closed embedding and property (5) follows from the one 
for $\alpha_{n,a}$. Assume $\fg=\frak{so}_{n=2m}$. 
Then from the diagram~\eqref{restricted diagram} we see that the involution $\alpha_{a}$
preserves the fibers of  
$\iota_\fc\circ\chi:\fg'\to\fc_{\fp,\bbR}\to\fc_{n,\bbR}$. 
Let $c=(c_2,c_4,...,\tilde c_n)\in\fc_{\fp,\bbR}$.
According to~\eqref{Pf}, if $\tilde c_n=0$ then $\chi^{-1}(c)=(\iota_\fc\circ\chi)^{-1}(\iota_\fc(c))$
and if 
$\tilde c_n=0$ then $(\iota_\fc\circ\chi)^{-1}(\iota_\fc(c))= 
\chi^{-1}(c)\sqcup\chi^{-1}(c')$
where $c'=(c_2,c_4,...,c_{n-2},-\tilde c_n)$.
In the first case, $\chi^{-1}(c)$ is equal to a fiber of 
$\iota_\fc\circ\chi$ 
and hence is invariant under $\alpha_a$.
Consider the second case.
Since $\chi^{-1}(c)$ contain a vector in $\fa_\bbR$ and  
$\alpha_0(M)=M$ for $M\in \fa_\bbR$, it follow that 
 $\alpha_0(\chi^{-1}(c))=\chi^{-1}(c)$.
 Since 
$\chi^{-1}(c)$ and $\chi^{-1}(c')$ are connected components of $(\iota_\fc\circ\chi)^{-1}(\iota_\fc(c))$
we must have $\alpha_a(\chi^{-1}(c))=\chi^{-1}(c)$ for all $a\in[0,1]$.
We are done.
This finishes  the proof of Theorem \ref{family of involutions for g}.

%%%%%%%%%%%%%%%%%%%
\quash{
Here is the explicit formula for $\alpha_s$ for each classical types:
 \\
(a) $\frakg=\frak{sl}_{n}(\bC)$:
 \begin{enumerate}
\item 
$\frakg_\bbR=\frak{sl}_n(\bbR), \alpha_s=\alpha_{n,s}$.
\item
$\frakg_\bbR=\frak{sl}_m(\mathbb H),\alpha_s=\alpha_{n,s}\circ\on{Ad}(S_m)\ \ (n=2m)$.
\item 
$\frakg_\bbR=\frak{su}_{p,n-p}, \alpha_s=\alpha_{n,s}\circ\on{Ad}I_{p,m-p}\circ\omega$
\end{enumerate}
(b) $\frakg=\frak{so}_{n}(\bC)$:
\begin{enumerate}
\item 
$\frakg_\bbR=\frak{so}_{p,n-p}, \alpha_s=\alpha_{n,s}\circ\on{Ad}I_{p,n-p}|_{\calN}$.
\item
$\frakg_\bbR=\frak u^*_m(\mathbb H), \alpha_s=\alpha_{n,s}\circ\on{Ad}S_m|_{\calN}\ \ (n=2m)$.
\end{enumerate}
(c) $\frakg=\frak{sp}_{2m}(\bC)$:
\begin{enumerate}
\item 
$\frakg_\bbR=\frak{sp}_{2m}(\bbR), \alpha_s=\alpha_{2m,s}|_\calN$.
\item
$\frakg_\bbR=\frak{sp}_{p,m-p}, \alpha_{s}=\alpha_{2m,s}\circ\on{Ad}K_{p,m-p}\circ\omega|_\calN$.
\end{enumerate}
}
%%%%%%%%%%%%
\subsection{Proof of Theorem \ref{homeomorphism for g}}
The proof is similar to the one of Theorem \ref{homeomorphism for g_n}.
Since 
$\fg=(\fg_n)^\sigma$ is the fixed-points subspace of the involution 
$\sigma$ on $\fg_n$ and the stratum of the Lusztig stratification of $\fg_n$ are invariant under $\sigma$
(the stratum are invariant under the adjoint action and transpose), we obtain a 
stratification of 
$\fg$ given by the 
$\sigma$-fixed points of the strata. 
The stratification on $\fg$ induces a stratification on $\fg'=\fg\times_{\fc}\fc_{\fp,\bbR}$, moreover,
the intersection of each stratum with the fibers of 
$\fg'\to\fc_{\fp,\bbR}$, if non-empty, is a finite union of $G$-orbits.

Let $\lvert\lvert-\lvert\lvert_{\fg'}:\fg'\to\bbR_{\geq0}$ be the restriction of the function
$\lvert\lvert-\lvert\lvert_{\fg_n'}$ to $\fg'\subset\fg'_n$ in the proof of Theorem \ref{homeomorphism for g_n}.
 It follows from Theorem \ref{family of involutions for g} and the construction of the function 
 $\lvert\lvert-\lvert\lvert_{\fg_n'}$ that
 the real analytic map $\fg'\to\fc_{\fp,\bbR}$ together with the stratification of $\fg'$ described above and  
the function $\lvert\lvert-\lvert\lvert_{\fg'}$ satisfy the assumption in Lemma \ref{fixed points}, and hence we obtain a stratified 
$K_\bbR$-equivariant homeomorphism 
\beq\label{key iso}
\fg_\bbR'=(\fg')^{\alpha_0}\to\fp'=(\fg')^{\alpha_1}
\eeq
which 
are real analytic on each stratum
and compatible with the maps to $\fc_{\fp,\bbR}$.
Since each stratum in $\fg'_\bbR$ (resp. $\fp'$)is a finite union of $G_\bbR$-orbits (resp. $K$-orbits)
and $K_\bbR$-acts simply transitively on connected components 
of each orbits, it follows that 
the homeomorphism~\eqref{key iso} 
restricts to a real analytic isomorphism between individual $G_\bbR$-orbits and $K$-orbits.
The proof of Theorem \ref{homeomorphism for g} is complete.

\section{Real and symmetric Springer theory}\label{applications}

\subsection{The real Grothendieck-Springer map}
Let $A_\bbR=\exp\fa_\bbR$ which is a 
closed, connected, abelian, diagonalizable subgroup of $G_\bbR$. 
Let $(\Phi,\fa_\bbR^*)$ be the root system (possible non-reduced)
of $(\fg_\bbR,\fa_\bbR)$. 
For each $\alpha\in\Phi$
we denote by $\fg_{\bbR,\alpha}\subset\fg_\bbR$ the corresponding $\alpha$-eigenspace. 
Choose a system of simple roots $\Delta=\{\alpha_1,...,\alpha_r\}\subset\Phi$ and denote 
by $\Phi^+$ (resp. $\Phi^-$) the corresponding set of positive roots (resp. negative roots).
We have the following decomposition:
\[\fg_\bbR=\frak m_\bbR\oplus\fa_\bbR\oplus\frak n_{\bbR}\oplus\bar{\frak n}_{\bbR}.\] 
where $\frak m_\bbR=Z_{\frak k_\bbR}(\fa_\bbR)$, $\frak n_\bbR=\oplus_{\alpha\in\Phi^+}\fg_{\bbR,\alpha}$, $\bar{\frak n}_\bbR=\oplus_{\alpha\in\Phi^-}\fg_{\bbR,\alpha}$. 

Let $\frak b_\bbR=\frak m_\bbR\oplus\fa_\bbR\oplus\frak n_\bbR$ be a minimal parabolic subalgebra of $\fg_\bbR$ and 
we denote by $B_\bbR=M_\bbR A_\bbR N_\bbR$ the corresponding minimal parabolic subgroup, here $N_\bbR=\exp(\frak n_\bbR)$ and
$M_\bbR=Z_{K_\bbR}(A_\bbR)$ is a group (possible not connected) with Lie algebra $\frak m_\bbR$. 
We write $F=\pi_0(M_\bbR)$.

An element $x\in \fg_\bbR$ is called semi-simple (resp. nilpotent)
if $\ad_x$ is diagonalizable over $\bbC$ (resp. nilpotent). An element  
$x\in \fg_\bbR$ is called hyperbolic (resp. elliptic) if 
it is semi-simple and 
the eigenvalues of 
$\ad_x$ are real (resp. purely imaginary).
For any $x\in \fg_\bbR$ we have the Jordan decomposition 
$x=x_e+x_h+x_n$ where $x_e$ is elliptic, $x_h$ is hyperbolic,
$x_n$ is nilpotent, and the three elements $x_e,x_h,x_n$ commute.

Consider the adjoint action of $G_\bbR$ on $\fg_\bbR$.
By a result of Richardson and Slodowy \cite{RS}, 
there exists a semi-algebraic set $\fg_\bbR//G_\bbR$ 
whose points are the 
semi-simple $G_\bbR$-orbits on $\fg_\bbR$. Furthermore,
there are maps $\chi_{\bbR}:\fg_\bbR\to \fg_\bbR//G_\bbR$ and
$\fg_\bbR//G_\bbR\to\fc$, such that 
the restriction of the 
Chevalley map $\chi:\fg\to\fc$ to $\fg_\bbR$ factors as 
\[\xymatrix{\fg_\bbR\ar[r]\ar[d]^{\chi_{\bbR}}&\fg\ar[d]^{\chi}\\
\fg_\bbR//G_\bbR\ar[r]&\fc}.\]
For any $x\in \fg_\bbR$ its image $\chi_{\bbR}(x)$ 
is given by the $G_\bbR$-orbit through the semi-simple part
$x_e+x_h$ of $x$.
We also have an embedding 
$\fa_\bbR//\rW\to \fg_\bbR//G_\bbR$, whose image consists of 
hyperbolic $G_\bbR$-orbits in $\fg_\bbR$,
such that the restriction of 
$\chi_{\bbR}$ to $\fa_\bbR$ factors as 
$\fa_\bbR\to \fa_\bbR//\rW\to \fg_\bbR//G_\bbR$.

Recall the subspace $\fg_\bbR'\subset \fg_\bbR$ 
consisting of elements in $\fg_\bbR$ with hyperbolic 
semi-simple parts (\ref{real eigenvalues}). 
By a result of Kostant \cite[Proposition 2.4]{Ko2}, any 
hyperbolic element $x$ in $\fg_\bbR$ is conjugate to an 
element in $\fa_\bbR$.
Moreover, the set of 
elements in $\fa_\bbR$ which are conjugate to $x$ is single 
$\rW$-orbit. It follows that the embedding $\fg_\bbR'\to \fg_\bbR$ factors through
an isomorphism
\beq
\fg_\bbR'=\fg_\bbR\times_{\fg_\bbR//G_\bbR}\fa_\bbR//\rW
\eeq
In particular, we have a natural projection map
\beq\label{real Chevalley}
\fg_\bbR'\to \fa_\bbR//\rW
\eeq
such that the composition 
$\fg_\bbR'\to \fa_\bbR//\rW\to\fc$
is equal to the map $\fg_\bbR'\to \fc_{\fp,\bbR}\subset\fc$
in~\eqref{projections}.

Introduce the \emph{real Grothendiek-Springer map}
\beq\label{GS map}
\widetilde\fg_\bbR=G_\bbR\times^{B_\bbR}\fb_\bbR\to \fg_\bbR\ \ \ \ (g,v)\to \on{Ad}_g(v).
\eeq
Note that unlike the complex case, the 
real Grothendiek-Springer map~\eqref{GS map} in general is not surjective.
Consider the base change of the real Grothendiek-Springer map to $\fg_\bbR'$:
\beq\label{hyperbolic GS}
\widetilde\fg_{\bbR}'\to \fg_\bbR'
\eeq
where $\widetilde\fg_{\bbR}'=\widetilde\fg_{\bbR}\times_{\fg_\bbR}\fg_\bbR'$.
By \cite[Proposition 2.5]{Ko2}, 
an element $x\in \fg_\bbR$ is in $\fg_\bbR'$ if and only if 
it is conjugate to an element in 
$\fa_\bbR+\frak n_\bbR$\footnote{In \emph{loc. cit.}, the claim is proved in the 
setting of adjoint action of $G_\bbR$ on $G_\bbR$. But the same argument works for the case of adjoint action of $G_\bbR$ on the Lie algebra 
$\fg_\bbR$, and hence $\fg_\bbR$.}. 
It follows that 
\[\widetilde\fg_{\bbR}'=G_\bbR\times^{B_\bbR}(\fa_\bbR+\frak n_\bbR)\]
and the map~\eqref{hyperbolic GS} is surjective. 
Moreover we have 
the following commutative diagram
\beq
\xymatrix{\widetilde\fg_{\bbR}'\ar[r]\ar[d]&\fg_\bbR'\ar[d]\\
\fa_\bbR\ar[r]&\fa_\bbR//\rW}
\eeq
where the map 
$\widetilde\fg_{\bbR}'\to \fa_\bbR$ is given by 
$(g,v=v_{a}+v_n)\to v_a$.

Consider the \emph{real Springer map}
\beq\label{RS map}
\pi_{\bbR}:\widetilde\calN_{\bbR}=G_\bbR\times^{B_\bbR}\frak n_\bbR\to \calN_\bbR
\eeq
We have the following cartesian diagrams
\beq
\xymatrix{\widetilde{\calN}_{\bbR}\ar[d]\ar[r]&\widetilde \fg_{\bbR}'\ar[d]\ar[r]&\widetilde\fg_{\bbR}\ar[d]\\
\calN_\bbR\ar[r]&\fg_\bbR'\ar[r]&\fg_\bbR}
\eeq
Since \eqref{hyperbolic GS} is surjective,
the real Springer map~\eqref{RS map} is also surjective.

\begin{lemma}\label{trivialization of real SR}
We have a $K_\bbR$-equivariant isomorphism 
$\widetilde\fg_{\bbR}'\is\widetilde\calN_{\bbR}\times \fa_\bbR$
commutes with projections to $\fa_\bbR$.
\end{lemma}
\begin{proof}
The Iwasawa decomposition
$G_\bbR= K_\bbR A_\bbR N_\bbR$ gives rise to $K_\bbR$-equivariant isomorphism
\[\widetilde\fg_{\bbR}'=G_\bbR\times^{B_\bbR}(\fa_\bbR+\frak n_\bbR)\is
K_\bbR\times^{M_\bbR}(\fa_\bbR+\frak n_\bbR).\]
Since $M_\bbR$ acts trivially on $\fa_\bbR$, we obtain 
\[\widetilde\fg_{\bbR}'\is (K_\bbR\times^{M_\bbR}\frak n_\bbR)\times \fa_\bbR.\]
On the other hand, we have 
\[\widetilde\calN_{\bbR}=G_\bbR\times^{B_\bbR}\frak n_\bbR\is K_\bbR\times^{M_\bbR}\frak n_\bbR.\]
Combining the isomorphisms above we get the desired 
$K_\bbR$-equivariant trivialization
\[\widetilde\fg_{\bbR}'\is\widetilde\calN_{\bbR}\times \fa_\bbR\]
commutes with projections to $\fa_\bbR$.
The proof is complete.
\end{proof}
\subsection{Sheaves of real nearby cycles}\label{Real NC}
Fix a point $a_{\bbR}\in \fa^{\on{rs}}_\bbR$ with image
$\xi_{\bbR}\in \fa_\bbR//\rW$.
Let $\mO_{\xi_{\bbR}}$ be the semi-simple $G_\bbR$-orbit through $a_\bbR$.
The centralizer  
$Z_{G_\bbR}(a_{\bbR})$ is isomorphic to $M_\bbR A_\bbR$
and it follows that the $G_\bbR$-equivariant fundamental group of 
$\mO_{\xi_{\bbR}}$
is isomorphic to $\pi_0(M_\bbR A_\bbR)\is\pi_0(M_\bbR)=F$.
For any one dimensional character $\chi$ of $F$ we denote by
$\mL_{\bbR,\chi}$ the $G_\bbR$-equivariant local system on 
$\mO_{\xi_{\bbR}}$ corresponding to 
$\chi$.

Consider the path 
$\gamma_{\bbR}:[0,1]\to \fa_\bbR//\rW$ given by
$\gamma_\bbR(s)=s\xi_{\bbR}$ and 
denote by
\[\calZ_{\bbR}=\fg_{\bbR}'\times_{\fa_\bbR//\rW}[0,1]\]
the base change of $\fg_\bbR'\to \fa_\bbR//\rW$~\eqref{real Chevalley} along 
$\gamma_{\bbR}$.
Note that $\gamma_{\bbR}$ is an embedding and hence 
$\calZ_{\bbR}$ is closed subvariety of $\fg_\bbR'$.
The fibers of the natural projection $f:\calZ_{\bbR}\to[0,1]$
over $0$ and $1$ are isomorphic to the nilpotent cone  $\calN_{\bbR}$ in $\fg_{\bbR}$ and 
semi-simiple orbit $\mO_{\xi_{\bbR}}$ respectively.
Moreover the $\bbR_{>0}$-action on $\fg_\bbR'$ induces a trivialization
\beq\label{Z_iR}
\mO_{\xi_{\bbR}}\times(0,1]\is\calZ_{\bbR}|_{(0,1]}\ \ \ (g,s)\to (sg,s). 
\eeq
Consider the following diagram 
\beq
\xymatrix{\mO_{\xi_{\bbR}}\times(0,1]\is\calZ_{\bbR}|_{(0,1]}\ar[r]^{\ \ \ \ \ \ \ \ \ u}\ar[d]&\calZ_{\bbR}\ar[d]^f&\calN_{\bbR}\ar[l]_{v}\ar[d]\\
(0,1]\ar[r]&[0,1]&\{0\}\ar[l]}
\eeq
where $u$ and $v$ are the natural embeddings.
Note that all the varieties in the diagram above carry natural $G_\bbR$-actions and all the maps 
between them are $G_\bbR$-equivariant.
Define the nearby cycles functor:
\beq
\Psi_{\bbR}:D_{G_\bbR}(\mO_{\xi_{\bbR}})\to D_{G_\bbR}(\calN_\bbR)\ \ \ \ \Psi_{\bbR}(\mF)=\psi_f(\mF\boxtimes\bC_{(0,1]})=v^*u_*(\mF\boxtimes\bC_{(0,1]}).
\eeq
For any character $\chi$ of $F$, consider the sheaf of nearby cycles with coefficient $\mL_\chi$
\beq
\mF_{\bbR,\chi}=\Psi_{\bbR}(\mL_\chi)
\eeq 
We will call $\Psi_{\bbR}$ the \emph{real nearby cycles functor} and  
$\mF_{\bbR,\chi}$ the \emph{sheaf of 
real nearby cycles}.

We shall give a formula of the nearby cycles sheaves in terms of the real Springer map
$\pi_{\bbR}:\widetilde\calN_{\bbR}\to \calN_\bbR$~\eqref{RS map}.
Since the $G_\bbR$-equivariant fundamental group of $G_\bbR/B_\bbR$, and hence that 
of $\widetilde\calN_{\bbR}$, 
is isomorphic to $\pi_0(B_\bbR)=\pi_0(M_\bbR)=F$, any 
character $\chi$ of $F$ gives rise to a 
$G_\bbR$-equivariant local system $\widetilde\mL_\chi$ on $\widetilde\calN_{\bbR}$.
Introduce the \emph{real Springer sheaf} 
\beq
\calS_{\bbR,\chi}=(\pi_{\bbR})_!\widetilde\mL_\chi.
\eeq

\begin{thm}\label{formula}
We have $\mF_{\bbR,\chi}\is\calS_{\bbR,\chi}$.

\end{thm}
\begin{proof}
Consider the path $\tilde\gamma_{\bbR}:[0,1]\to \fa_\bbR$
given by $\tilde\gamma_{\bbR}(s)=s(a_\bbR)$
and 
let \[\widetilde\calZ_{\bbR}=\tilde\fg_{\bbR}\times_{\fa_\bbR}[0,1]\]
to be the base change of the map
$\tilde\fg_{\bbR}\to \fa_\bbR$
along the path $\tilde\gamma_{\bbR}$.
The fiber of the projection $\tilde f:\widetilde\calZ_{\bbR}\to [0,1]$
over $0$ and $1$ are given by
$\widetilde\calN_{\bbR}$ and $\mO_{\xi_{\bbR}}$ respectively. 
Moreover,
there is a trivialization
\beq\label{tilde Z_iR}
\widetilde\calZ_{\bbR}|_{(0,1]}\is\mO_{\xi_{\bbR}}\times(0,1] \ \ \ ((g,v),s)\to (\on{Ad}_{g}(s^{-1}v),s)
\eeq
It follows that 
the real Grothendieck-Springer map $\widetilde\fg_{\bbR}\to \fg_\bbR$
restricts to a map $\tau_{\bbR}:\widetilde\calZ_{\bbR}\to\calZ_{\bbR}$
which is an isomorphism over $\calZ_{\bbR}|_{(0,1]}$. Consider he following commutative diagram
\beq
\xymatrix{\mO_{\xi_{\bbR}}\times(0,1]\ar[r]^{~\eqref{tilde Z_iR}}\ar[d]^{\id}&\widetilde\calZ_{\bbR}|_{(0,1]}\ar[r]^{\ \ \ \ \tilde u}\ar[d]^{\tau_{\bbR}}&\widetilde\calZ_{\bbR}\ar[d]^{\tau_{\bbR}}&\widetilde\calN_{\bbR}\ar[l]_{\tilde v}\ar[d]^{\pi_{\bbR}}\\
\mO_{\xi_{\bbR}}\times(0,1]\ar[r]^{~\eqref{Z_iR}}\ar[dr]&\calZ_{\bbR}|_{(0,1]}\ar[r]^{\ \ \ \ u}\ar[d]&\calZ_{\bbR}\ar[d]&\calN_{\bbR}\ar[l]_{v}\ar[d]\\
&(0,1]\ar[r]&[0,1]&\{0\}\ar[l]}
\eeq
Consider the nearby cycles functor 
\[\widetilde\Psi_{\bbR}:D_{G_\bbR}(\mO_{\xi_{\bbR}})\to D_{G_\bbR}(\widetilde\calN_{\bbR})\ \ 
\ \ \ \ \widetilde\Psi_{\bbR}(\mF)=\tilde v^*\tilde u_*(\mF\boxtimes\bC_{(0,1]})\]
Since $\tau_{\bbR}$ is proper and $(\tau_{\bbR})_!(\mF\boxtimes\bC_{(0,1]})\is\mF\boxtimes\bC_{(0,1]}$, the proper base change for nearby cycles functors implies that 
there is a canonical isomorphism
\beq\label{base change}
(\pi_{\bbR})_!\widetilde\Psi_{\bbR}(\mF)=
(\pi_{\bbR})_!\psi_{\tilde f}(\mF\boxtimes\bC_{(0,1]})\is\psi_f((\tau_{\bbR})_!(\mF\boxtimes\bC_{(0,1]}))\is\psi_f(\mF\boxtimes\bC_{(0,1]})=\Psi_{\bbR}(\mF).
\eeq
On the other hand, the $K_\bbR$-equivariant trivialization  
in Lemma \ref{trivialization of real SR} gives rise to a $K_\bbR$-equivariant isomorphism 
\beq\label{trivialization}
\widetilde\calZ_{\bbR}\is\widetilde\calN_{\bbR}\times[0,1]
\eeq
commutes with projections to $[0,1]$. In addition, 
there exits a $K_\bbR$-equivariant isomorphism 
$q:\widetilde\calN_{\bbR}\is\mO_{\xi_{\bbR}}$ such that 
$q^*\mL_\chi\is\widetilde\mL_\chi$ and making the following diagram commute
\[\xymatrix{\calZ_{\bbR}|_{(0,1]}\ar[r]^{~\eqref{trivialization}\ \ \ \ }\ar[d]^\id&\widetilde\calN_{\bbR}\times(0,1]\ar[d]^{q\times\id}\\
\widetilde\calZ_{\bbR}|_{(0,1]}\ar[r]^{\eqref{tilde Z_iR}\ \ \ \ }&\mO_{\xi_{\bbR}}\times(0,1]}.\]
It follows that 
\beq\label{L=L}
\widetilde\Psi_{\bbR}(\mL_{\chi})\is\psi_{\tilde f}(\mL_{\chi}\boxtimes\bC_{(0,1]})\is
\psi_{\tilde f}(\widetilde\mL_{\chi}\boxtimes\bC_{(0,1]})\is\widetilde\mL_\chi
\eeq
as object in $D_{K_\bbR}(\widetilde\calN_{\bbR})$.
Since $D_{G_\bbR}(\widetilde\calN_{\bbR})\subset D_{K_\bbR}(\widetilde\calN_{\bbR})$
is a full subcategory (as $G_{\bbR}/K_{\bbR}$ is contractible), 
we conclude that 
\[\calS_{\bbR,\chi}=(\pi_{\bbR})_!\widetilde\mL_\chi\stackrel{~\eqref{L=L}}\is(\pi_{\bbR})_!\widetilde\Psi_{\bbR}(\mL_\chi)\stackrel{~\eqref{base change}}\is\Psi_{\bbR}(\mL_\chi)=\mF_{\bbR,\chi}\in D_{G_\bbR}(\calN_{\bbR})\]
The proof is complete.
\end{proof}

\subsection{Sheaves of symmetric nearby cycles}\label{symmetric NC}

The discussion in the previous subsection has a counterpart in the setting of symmetric space.
Recall the subspace $\fp'\subset\fp$ consisting of elements 
$x$ in $\fp$ such that the eigenvalues of $\ad_x$ are real.
In \cite{KR}, Kostant and Rallis proved that for any such $x$, its semi-simple part 
$x_s\in\fp$ 
is conjugate to an element in 
$\fa_\bbR$, moreover, the set of elements in $\fa_\bbR$ which are conjugate to $x_s$ is single 
$\rW$-orbit. It follows that the subspace  
$\fp'$ is equal to the base change 
\[\fp'=\fp\times_{\fc_\fp}\fa_\bbR//\rW\]
of $\chi_\fp:\fp\to\fc_\fp$ along $\fa_\bbR//\rW\subset\fc_\fp$.

Let $a_\fp\in \fa_\bbR^{\on{rs}}$ with image $\xi_{\fp}\in \fa_\bbR//\rW$.
Let $\mO_{\xi_\fp}$ be the $K$-orbit through $a_\fp$.
We have $Z_K(a_\fp)=MA$ and it follows that the $K$-equivariant fundamental group 
of $\mO_{\xi_\fp}$ is isomorphic to
$\pi_0(Z_K(a_\fp))=\pi_0(MA)=\pi_0(M)=F$.
For any character $\chi$ of $F$ we denote by
$\mL_{\fp,\chi}$ the $K$-equivariant local system on $\mO_{\xi_\fp}$.
Consider the path 
$\gamma_\fp:[0,1]\to \fa_\bbR//\rW$ given by
$\gamma_\fp(s)=s\xi_{\fp}$ and define 
\[\calZ_{\fp}=\fp'\times_{\fa_\bbR//\rW}[0,1]\] 
The fibers of the natural projection $f_\fp:\calZ_{\fp}\to[0,1]$
over $0$ and $1$ are isomorphic to the nilpotent cone  $\calN_{\fp}$ in $\fp$ and 
the $K$-orbit $\mO_{\fp}$.
Moreover the $\bbR_{>0}$-action on $\fp'$ induces a trivialization
\beq\label{Z_p}
\mO_{\xi_\fp}\times(0,1]\is\calZ_{\fp}|_{(0,1]}\ \ \ (g,s)\to (sg,s). 
\eeq
Consider the following diagram 
\beq
\xymatrix{\mO_{\xi_{\fp}}\times(0,1]\is\calZ_{\fp}|_{(0,1]}\ar[r]^{\ \ \ \ \ \ \ \ \ u}\ar[d]&\calZ_{\fp}\ar[d]^{f_\fp}&\calN_{\fp}\ar[l]_{v}\ar[d]\\
(0,1]\ar[r]&[0,1]&\{0\}\ar[l]}
\eeq
where $u$ and $v$ are the natural embeddings.
Note that all the varieties in the diagram above carry natural $K$-actions and all the maps 
between them are $K$-equivariant.
Introduce the nearby cycles functor:
\beq
\Psi_{\fp}:D_{K}(\mO_{\xi_\fp})\to D_{K}(\calN_\fp)\ \ \ \ \Psi_{\fp}(\mF)=\psi_{f_\fp}(\mF\boxtimes\bC_{(0,1]})=v^*u_*(\mF\boxtimes\bC_{(0,1]}).
\eeq
For any character $\chi$ of $F$, consider the nearby cycles sheaf with coefficient $\mL_{\fp,\chi}$
\beq
\mF_{\fp,\chi}=\Psi_{\fp}(\mL_{\fp,\chi})
\eeq 
We will call $\Psi_{\fp}$ the \emph{symmetric nearby cycles functor} and 
$\mF_{\fp,\chi}$ the
\emph{sheaf of symmetric nearby cycles}.

Recall the $K_\bbR$-equivariant stratified homeomorphism 
\beq\label{key homeo}
\fg'_\bbR\is\fp'
\eeq in Theorem~\eqref{homeomorphism for g}. 
Since 
the homeomorphism 
~\eqref{key homeo} 
commutes with projection to 
$\fc_{\fp,\bbR}$ and the natural map 
$\fa_\bbR//\rW\to\fc_{\fp,\bbR}$ is a finite map\footnote{Recall that 
$\fc_{\fp,\bbR}$ is by definition 
the image of the map
$\fa_\bbR\to\fa//\rW=\fc_\fp\to\fc$.
Since the later map
$\fc_\fp\to\fc$ is in general not a closed embedding, 
the map $\fa_\bbR//\rW\to\fc_{\fp,\bbR}$ is not a closed embedding in general.
}, 
for any $\xi_{\bbR}\in \fa_\bbR^{\on{rs}}/\rW$
there exists a unique $\xi_\fp\in \fa^{\on{rs}}_\bbR//\rW$ such that 
~\eqref{key homeo}
restricts to a $K_\bbR$-equivariant real analytic isomorphism between individual 
fibers 
\[\mO_{\xi_{\bbR}}\is\mO_{\xi_\fp}.\]
Since~\eqref{key homeo} is $\bbR_{>0}$-equivariant, 
the isomorphism above and
the trivializations~\eqref{Z_iR} and~\eqref{Z_p}
imply that~\eqref{key homeo}
induces a $K_\bbR\times\bbR_{>0}$-equivariant 
homeomorphism
\beq
\calZ_{\bbR}\is\calZ_\fp
\eeq
commutes with projections to $[0,1]$.
The homeomorphism above gives rise to a canonical commutative square of
functors
\beq\label{square}
\xymatrix{D_{G_\bbR}(\mO_{\xi_{\bbR}})\ar[r]^{\Psi_{\bbR}}\ar[d]&D_{G_\bbR}(\calN_\bbR)\ar[d]\\
D_K(\mO_{\xi_\fp})\ar[r]^{\Psi_\fp}&D_K(\calN_\fp)}
\eeq
where the upper and lower arrows are the real and symmetric nearby cycles respectively and 
the vertical arrows are the 
equivalences in~\eqref{eq: intro equiv}.
Since the equivalence $D_{G_\bbR}(\mO_{\xi_{\bbR}})\is D_{K}(\mO_{\xi_{\fp}})$
maps $\mL_{\bbR,\chi}$ to $\mL_{\fp,\chi}$, 
the diagram~\eqref{square} and Theorem \ref{formula}
imply the following:

\begin{thm}\label{S=R}
Assume $\fg$ is of classical type.
Under the equivalence $D_K(\calN_\fp)\is D_{G_\bbR}(\calN_{\bbR})$ in~\eqref{eq: intro equiv},
the sheaf of symmetric nearby cycles $\mF_{\fp,\chi}$ becomes the
the sheaf of real nearby cycles $\mF_{\bbR,\chi}$, which is also isomorphic to the 
real Springer sheaf $\mS_{\bbR,\chi}$.
In particular,
the real Springer map 
$\pi_\bbR:\widetilde\calN_\bbR\to\calN_\bbR$ is a semi-small map and the 
real Springer sheaf $\calS_{\bbR,\chi}$
is a perverse sheaf.

\end{thm}

\quash{
\begin{remark}
If $\fp$ contains a regular nilpotent element in $\fg$, for example, when 
$\fg_\bbR$ is quasi-split, then by \cite[Theorem 3.6]{P}
the natural map
$\fc_\fp\to\fc$
is a closed embedding and we have 
$\fc_{\fp,i\bbR}=\fc_{\fp,i\bbR}'$.
\end{remark}
}

{}

\end{document}